\documentclass[12pt]{amsart}
\usepackage{amsmath,amssymb,euscript,mathrsfs,pstricks, eufrak}
\usepackage[small,heads=LaTeX,nohug]{diagrams}
\usepackage{hyperref}
\usepackage[all]{xy}
\usepackage{enumerate}

\pushQED{\qed}

\psset{unit=1pt}
\psset{arrowsize=4pt 1}
\psset{linewidth=.5pt}

\newcommand{\define}{\textbf}
\newcommand{\comment}{$\star$ \texttt}

\renewcommand{\setminus}{\smallsetminus}
\renewcommand{\oplus}{\bigoplus}
\renewcommand{\phi}{\varphi}

\renewcommand{\tilde}{\widetilde}
\renewcommand{\hat}{\widehat}

\renewcommand{\L}{\mathbb{L}}

\newcommand{\D}{\mathbb{D}}
\newcommand{\C}{\mathbb{C}}
\newcommand{\Q}{\mathbb{Q}}
\newcommand{\R}{\mathbb{R}}
\newcommand{\N}{\mathbb{N}}
\newcommand{\Z}{\mathbb{Z}}
\renewcommand{\P}{\mathbb{P}}
\newcommand{\A}{\mathbb{A}}
\renewcommand{\O}{\mathcal{O}}
\newcommand{\T}{\mathcal{T}}

\newcommand{\M}{\mathbb{M}}
\newcommand{\K}{\mathbb{K}}
\newcommand{\G}{\mathbb{G}}
\newcommand{\cF}{\mathcal{F}}
\newcommand{\bk}{\mathbf{k}}
\newcommand{\tri}{\triangle}
\def\<{\ensuremath{\langle}}
\def\>{\ensuremath{\rangle}}
\newcommand{\Tr}{{(X^\circ\!, \, \Sigma)}}
\newcommand{\TR}{{(X^\circ\!, \, \Sigma')}}
\newcommand{\Trr}{{(X^\circ\!, \, \triangle)}}

\newcommand{\Gr}{Gr}

\newcommand{\X}{\mathcal{X}}

\DeclareMathOperator{\an}{an}

\DeclareMathOperator{\codim}{codim}

\DeclareMathOperator{\im}{im}

\DeclareMathOperator{\Hom}{Hom}
\DeclareMathOperator{\Spec}{Spec}
\DeclareMathOperator{\ord}{ord}

\DeclareMathOperator{\pt}{pt}
\DeclareMathOperator{\lk}{lk}

\DeclareMathOperator{\Int}{Int}

\DeclareMathOperator{\Trop}{Trop}

\DeclareMathOperator{\prim}{prim}
\DeclareMathOperator{\coker}{coker}

\DeclareMathOperator{\Var}{Var}
\DeclareMathOperator{\In}{in}
\DeclareMathOperator{\Span}{Span}
\DeclareMathOperator{\Star}{Star}
\DeclareMathOperator{\init}{in}

\DeclareMathOperator{\ver}{vert}
\DeclareMathOperator{\vol}{vol}
\DeclareMathOperator{\gen}{gen}

\newtheorem{theorem}{Theorem}[section]
\newtheorem{lemma}[theorem]{Lemma}
\newtheorem{proposition}[theorem]{Proposition}
\newtheorem{corollary}[theorem]{Corollary}

\newtheorem*{ntheorem}{Theorem}

\theoremstyle{definition}
\newtheorem{definition}[theorem]{Definition}
\newtheorem{remark}[theorem]{Remark}
\newtheorem{example}[theorem]{Example}

\newcommand{\excise}[1]{}

\begin{document}

\title[]{
Tropical Geometry and the Motivic Nearby Fiber
}
\author{Eric Katz}
\address{Department of Mathematics\\University of Texas-Austin\\Austin, TX 78712-0257}
\email{eekatz@math.utexas.edu}
\author{Alan Stapledon}
\address{Department of Mathematics\\University of British Columbia\\ BC, Canada V6T 1Z2}
\email{astapldn@math.ubc.ca}
\subjclass{14T05 (primary), 14D05 (secondary)}

\keywords{tropical geometry, monodromy, Hodge theory}
\date{}
\thanks{}

\begin{abstract}
We 
construct motivic invariants of a subvariety 
of an algebraic torus from its tropicalization and initial degenerations.
More specifically, we introduce an invariant of a compactification 
of such a variety called the ``tropical motivic nearby fiber.''  This invariant specializes in the sch\"on case to the Hodge-Deligne polynomial of the limit mixed Hodge structure of a corresponding degeneration.  We give purely combinatorial expressions for this Hodge-Deligne polynomial in the cases of sch\"on hypersurfaces and matroidal tropical varieties. We also deduce a formula for the Euler characteristic of a general fiber of the degeneration. 
\end{abstract}

\maketitle
 
\tableofcontents

\section{Introduction}\label{s:intro}

Let $\K$ be a field with a non-Archimedean discrete valuation $v:\K^*\rightarrow\Z$ with valuation ring $\O$ and residue field the complex numbers $\C$.  Tropicalization is a procedure which assigns to a $d$-dimensional subvariety $X^\circ$ of the torus $(\K^*)^n$, a polyhedral complex $\Trop(X^\circ)\subset\R^n$ of pure dimension $d$, such that the rational points of $\Trop(X^\circ)$ parametrize the `interesting' initial degenerations $\init_w X^\circ$ of $X^\circ$. It is a natural question to ask which geometric properties of $X^\circ$ can be recovered from 
the combinatorial data $\Trop(X^\circ)$, together with the algebraic geometric data of the initial degenerations.


Consider a pair $(X^\circ, \Sigma)$, where $\Sigma$ is a polyhedral structure on $\Trop(X^\circ)$ that can be extended to a polyhedral subdivision of $\R^n$ and which is \emph{tropical} in the sense of Tevelev (see Definition~\ref{d:tropical}). We introduce a new invariant 
$\psi_\Trr$, called the \define{tropical motivic nearby fiber}, which lies in the Grothendieck ring $K_0(\Var_\C)$ of complex algebraic varieties. More precisely, $K_0(\Var_\C)$ 
is the free $\Z$-module generated by isomorphism classes $[V]$ of complex varieties $V$, modulo  the relation 
$[V] = [U] + [V \setminus U]$, 
whenever $U$ is an open subvariety of $V$. 
If we set $\L := [\A^1]$, then the tropical motivic nearby fiber is defined by
\[
\psi_{\Trr} =  \sum_{F \in \Sigma} \, [X^\circ_F] \, (1 - \L)^{\dim F - \dim \tau_F}, 
\]
where  $\In_w X^\circ \cong X^\circ_F \times (\C^*)^{\dim F}$ for any $w$ in the relative interior of $F$, 
and  $\tau_F$ denotes the recession cone of $F$. Our first main result (Theorem~\ref{t:independence}) states that the  tropical motivic nearby fiber is independent of the choice of polyhedral structure $\Sigma$, provided that   
the corresponding recession fan $\triangle = \{ \tau_F \mid F \in \Sigma \}$ is fixed.

For the remainder of the introduction, we will assume that $X^\circ$ is \emph{sch\"on} in the sense of Tevelev (see Definition~\ref{d:schon}), and that the recession fan $\tri$ is unimodular. Let $\O$ be  the ring of germs of analytic functions in $\C$ in a neighborhood of the origin,  with the valuation equal to the vanishing order of a function at the origin.
Then $X^\circ$ induces a a family of subvarieties of $(\C^*)^n$ over a punctured disc $\D^*$ about the origin, and the recession fan $\tri$ induces  a compactification $X^\circ \subseteq X$, and a smooth, proper map $f:  X \rightarrow \D^*$ (see Section~\ref{s:schon}). 
The \define{motivic nearby fiber} $\psi_f \in K_0(\Var_\C)$ of $f$ was introduced by Denef and Loeser \cite{DLGeometry}, and encodes information about the variation of  Hodge structure of a fiber $X_{\gen}$ of $f$. In fact, the motivic nearby fiber was introduced as a specialization of an invariant of a suitable extension of $f$ to a family over $\D$, called the motivic zeta function, which is defined using motivic integration. 
Our second main result is the following.

\begin{ntheorem}[Theorem~\ref{t:Schon}]
With the notation above, if $X^\circ$ is sch\"on and $\tri$ is a unimodular recession fan associated to the tropical variety $\Trop(X^\circ)$, then  $\psi_{(X^\circ, \tri)} = \psi_f$.
\end{ntheorem}

This result has a number of Hodge-theoretic and topological consequences. 
By the results of Steenbrink and others \cite{SteLimits}, the cohomology of 
a fixed non-zero fiber $X_{\gen}$ of $f$ carries a \define{limit mixed Hodge structure}.
In particular,  
the associated limit Hodge numbers $h^{p,q}(H^m (X_\infty))$ refine the Hodge numbers $h^{p, m - p} (X_{\gen})$ of $X_{\gen}$ \eqref{e:Hodge}, and record the sizes of the Jordan blocks of the logarithm of the monodromy operator on the cohomology of $X_{\gen}$. We may consider the ring homomorphism
\[
E: K_0(\Var_\C) \rightarrow \Z[u,v], \: \: E([V]) = E(V),
\]
which takes a complex variety $V$ to its Hodge-Deligne polynomial $E(V)$, and satisfies $E(\L) = uv$ (see Section~\ref{s:limit}). Composition with the homomorphism $\Z[u,v] \rightarrow \Z$ obtained by setting $u = v  = 1$, gives the homomorphism
\[
e: K_0(\Var_\C) \rightarrow \Z,
\]
which takes a complex variety $V$ to its topological Euler characteristic $e(V)$. 
With the notation of the theorem above, Corollary~\ref{c:topological} states that the tropical motivic nearby fiber specializes to
\[
E(\psi_{(X^\circ, \tri)} ) =  E(X_\infty;u,v)  := \sum_{p,q} \left(\sum_m (-1)^m h^{p,q}(H^m(X_{\infty})) \right) u^p v^q,
\]
and  the topological Euler characteristic  of $X_{\gen}$ is given by
\[
e(X_{\gen}) = e(\psi_{(X^\circ, \tri)} ) =  \sum_{\substack{F \in \Sigma \\ \dim F = \dim \tau_F}} \, e(X_{F}^\circ). 
\]

The parameterizing complex $\Gamma_{X^\circ}$ of $\Trop(X^\circ)$ is a polyhedral complex which was introduced in \cite{HelmKatz} and admits a natural map $p:\Gamma_{X^\circ}\rightarrow\Trop(X^\circ)$ (see Section~\ref{s:tropical}).  Continuing with the notation of the theorem above, 
 in Corollary \ref{c:Betti}, we show that the Betti numbers of $\Gamma_{X^\circ}$ satisfy
\[b_m(\Gamma_{X^\circ}) \le \min_{p + q = m} h^{p,q}(X_{\gen}). \]
This strengthens the result of Helm and the first named author \cite[Corollary~5.8]{HelmKatz} that 
\[
b_m(\Gamma_{X^\circ}) \le \frac{b_m(X_{\gen})}{m + 1}. 
\]
We refer the reader to Corollary \ref{c:genus} for a further upper bound on $b_d(\Gamma_{X^\circ})$.

Finally, to use the tropical motivic nearby fiber to compute the polynomial $E(X_\infty;u,v)$, 
one needs to know the Hodge-Deligne polynomials of the complex varieties $X^\circ_F$.
In addition to the easy case where $X^\circ$ is a curve, we explore two cases where this is possible: when $X^\circ$ is a sch\"on hypersurface and when $\Trop(X^\circ)$ is a smooth tropical variety.  

In the case when $X^\circ$ is a hypersurface, the varieties $X^\circ_F$ are  sch\"on hypersurfaces in complex tori. 
Their Hodge-Deligne numbers are determined by their Newton polytopes by a result of Danilov and Khovanskii \cite{DKAlgorithm}, and hence one obtains a combinatorial formula for $E(X_\infty;u,v)$.  In fact, the tropical variety is dual to a regular, lattice polyhedral decomposition $\mathcal{T}$ of the Newton polytope $P$ of $X^\circ$ in $\R^{d + 1}$, and one obtains combinatorial formulas for the limit Hodge numbers $h^{p,q}(H^m (X_\infty))$ in terms of $\mathcal{T}$ (Corollary~\ref{c:formula}). In particular, Corollary~\ref{c:p0} states that 
 for $p > 0$, 
\[
h^{p,0} (H^{d}(X_\infty)) =  \sum_{  \substack{  Q \in \T^{(0)} \\ \dim Q = p + 1}}  \# (\Int(Q) \cap \Z^n),
\]
where $\T^{(0)}$ denotes the faces of $\mathcal{T}$ whose relative interior lies in the relative interior of $P$. Also, 
\[
h^{0,0} (H^{d}(X_\infty)) =  \sum_{  \substack{  Q \in \T^{(0)} \\ \dim Q \le 1}}  \# (\Int(Q) \cap \Z^n).
\]

Matroidal tropical varieties are polyhedral complexes locally modeled on the matroid fans of Ardila and Klivans \cite{AKBergman}.  There were originally called smooth tropical varieties by Mikhalkin \cite{MikhRat}.  If $\Trop(X^\circ)$ is a matroidal tropical variety, then $X^\circ$ is sch\"on.  Moreover, the initial degenerations of $X^\circ$ are all intersections of linear subspaces of $\P^n$ with $(\C^*)^n$.   Their Hodge-Deligne numbers are determined by the matroid of that linear space due to the work of Orlik and Solomon \cite{OSHyperplanes}. Hence we determine a combinatorial formula for
 the tropical motivic fiber $\psi_{\Trr}$ (Corollary~\ref{c:sformula}).
 As an immediate consequence, we deduce the following formula for the Euler characteristic of a general fiber (Corollary~\ref{c:sEuler}),
 \[e(X_{\gen})= \sum_{\substack{F\in\Sigma\\ \dim(F)=\dim(\tau_F)}} (\chi_{_{\M_F}}(1)+\chi'_{_{\M_F}}(1)),\]
 where $\chi_{_{\M_F}}(q)$ denotes the characteristic polynomial of the matroid $\M_F$ associated to the star quotient of $F$ (see Section~\ref{s:smooth}), and
$\chi'_{_{\M_F}}(q)$ denotes the derivative of  $\chi_{_{\M_F}}(q)$. 
 
We mention some related work.  In \cite{GS}, Gross and Siebert construct 
a scheme $X_0$ from certain combinatorial data.   If $X_0$ is 
embedded in a family $\X$ over $\O$, they  determine the limit mixed Hodge structure 
in terms of the combinatorial data.  In \cite{Ruddat}, Ruddat gives a spectral sequence for determining the logarithmic Hodge groups of toric log Calabi-Yau spaces of hypersurface type in terms of tropical degeneration data and Jacobian rings.  In \cite{HelmKatz}, Helm and the first named author relate the topology of the parameterizing complex $\Gamma_{X^\circ}$ to the monodromy filtration on $H^*(X_{\gen})$ for sch\"on subvarieties $X^\circ\subset (\K^*)^n$.

\medskip
\noindent
{\it Notation and conventions.}  
All schemes are over the complex numbers.  For us, a variety is a  not necessarily irreducible, separated reduced scheme of finite type over $\C$.
Cohomology will be taken with $\Q$ coefficients, with respect to the usual (complex) topology.

\medskip
\noindent
{\it Acknowledgements.} The authors would like to thank Sam Payne for his talk `Topology of compactified tropicalizations' at MSRI, October 12, 2009. It provided a great deal of inspiration for this project. The authors would also like to thank Kalle Karu for helping us to resolve a long standing technical difficulty concerning semi-stable reduction for toric varieties. The authors also benefited greatly from the work and help of David Helm.
This paper arose out of discussions during working group meetings on the Topology of Tropical Varieties held at the Mathematical Sciences Research Institute during the Fall 2009 Program on Tropical Geometry.  We would like to thank the organizers of this program and the MSRI staff for making such a program possible.

\section{Tropical geometry}\label{s:tropical}

We review some notions for studying tropical geometry in the non-constant coefficient case.  Many of these were developed by Tevelev \cite{TevComp} and Hacking \cite{Hacking} in the constant coefficient case and extended to the non-constant coefficient case by Luxton-Qu \cite{LuxtonQu}, Qu \cite{QuToric} and Helm-Katz \cite{HelmKatz}.
Let $\K$ be a field with a non-Archimedean discrete valuation $v:\K^*\rightarrow\Z$.  Let $\O$ be its valuation ring with residue field $\bk$.  The examples to keep in mind are 
\begin{enumerate}
\item $\K=\C((t))$, $\O=\C[[t]]$, and $\bk=\C$,
\item $\K=\C(t)$, $\O=\C[t]_{(t)}$, and $\bk=\C$,
\item  $\O$ is the ring of germs of analytic functions in $\C$ in a neighborhood of the origin, $\K$ is its field of fractions, and $\bk=\C$.
\end{enumerate}
 We suppose for ease of exposition that there is a section of the extension of the valuation $v:\overline{\K}^*\rightarrow\Q$ given by $\Q\ni w\mapsto t^w\in\overline{\K}^*$.  Let $\T \cong \G_m^n$ be a split $n$-dimensional torus over $\O$, and let
$T = \T \times_{\O} \K$ be the corresponding torus over $\K$.  Let $N$ be the cocharacter lattice, $N=\Hom(\K^*,T)$ and $N_\R=N\otimes\R$. For $w=(w_1,\dots,w_n)\in \Q^n$, write $t^w=(t^{w_1},\dots,t^{w_n})\in T$.

Let $X^\circ$ be a subvariety of $T$ and $w\in \frac{1}{M}\Z^n$, let $\K'=\K[[t^\frac{1}{M}]]$.  Then the initial degeneration, $\init_w X^\circ\subset (\bk^*)^n$ is given by
\[\init_w X^\circ=\overline{t^{-w} X^\circ_{\K'}}\times_{\O'} \bk\]
where the closure is taken in $\T$. 

\begin{definition} The tropical variety $\Trop(X^\circ)$ associated to $X^\circ$ is the closure in $\R^n$ of the set of points
\[\{w\in\Q^n|\init_w X^\circ\neq \emptyset\}.\]
\end{definition}

For $X_\bk^\circ\subset T_\bk$, we define the tropicalization of $X_\bk$ by base-change.  Let $\K'=\bk((t))$ and $X_{\K'}^\circ=X^\circ \times_\bk \K'\subset (\K'^*)^n$.  Set $\Trop(X_\bk^\circ)=\Trop(X_{\K'}^\circ)$.  

Let $\Sigma$ be a subcomplex of a rational polyhedral subdivision of $\R^n$. One can define a toric scheme $\P(\Sigma)$ over $\O$ from $\Sigma$.  First, let $\tilde{\Sigma}$ be the union of the cones over faces of $\Sigma\times\{1\}$ in $N_\R\times \R_{\geq 0}$.  This forms a rational polyhedral fan in $\R^n \times \R_{\geq 0}$ by Corollary 3.12 of \cite{BGS}.   Let $\P(\tilde{\Sigma})$ be the corresponding toric variety.  The projection $N_\R\times\R_{\geq 0}\rightarrow \R_{\geq 0}$ induces a projection $\P(\tilde{\Sigma})\rightarrow\A^1$.  Let $\Spec \O\rightarrow \A^1=\Spec \Z[u]$ be induced by the universal homomorphism $\Z\rightarrow \O$ and $u\mapsto t$.  Then $\P(\Sigma)=\P(\tilde{\Sigma})\times_{\Z[u]} \O$.   Let $\Delta$ be the recession fan of $\Sigma$, that is the fan given by $\tilde{\Sigma}\cap (N_\R\times \{0\})$.  The generic fiber of $\P(\Sigma)$ is the toric variety $\P(\Delta)$.  For cones $\sigma\in\Delta$, let $U_\sigma$ be the corresponding torus orbit of $\P(\Delta)$.   
See \cite[Section~2]{HelmKatz} for a more detailed description of the relevant construction.

Faces $F\in\Sigma$ correspond to torus orbits $U_F$ contained in the central fiber of $\P$.  There is an inclusion reversing correspondence $F\mapsto \overline{U}_F$ between faces and orbit closures.
The torus orbits $U_F$ give a decomposition of the central fiber of $\P$.  We define $T_F\subset T_\bk$ to be the torus fixing $U_F$ pointwise and $N_F$ to be the cocharacter lattice of $T_F$.  Note that $N_F=\Span_\R(F-w)\cap N$ where $w\in \Int(F)$.  

Let $\X$ denote the closure of $X^\circ$ in $\P(\Sigma)$, and let $X$ be the generic fiber of $\X$.  

\begin{definition}\label{d:tropical} The pair $(X^\circ,\P(\Sigma))$ is {\em tropical} if the map
$$m: \T \times_ {\O} \X \rightarrow \P(\Sigma)$$
is faithfully flat, and $\X \rightarrow \O$ is proper.  
\end{definition}

We then have the following, due to Tevelev \cite{TevComp} in the constant coefficient case.
In the non-constant coefficient case they can be found in \cite{LuxtonQu}.

\begin{proposition}\cite[Proposition~6.8]{LuxtonQu} If $X^\circ$ is not invariant under any torus in $T$, then there is a toric scheme $\P(\Sigma)$ such that $(X^\circ,\P(\Sigma))$ is tropical.  
\end{proposition}

In addition, $\Sigma$ in the above proposition can be chosen to be a subcomplex of a rational polyhedral subdivision of $\R^n$.
Moreover, tropical pairs are stable under refinement according to the following:

\begin{proposition}\label{p:tropical} \cite[Proposition~6.7]{LuxtonQu} Suppose $(X^\circ,\P(\Sigma))$ is tropical and $\P(\Sigma')\rightarrow\P(\Sigma)$ is a morphism of toric schemes corresponding to a refinement $\Sigma'$ of $\Sigma$.  Then $(X^\circ,\P(\Sigma'))$ is tropical.  Moreover, 
$\X'=\overline{X^\circ}\subset \P(\Sigma')$ is the inverse image of $\X$. 
\end{proposition}  

If $(X^\circ,\P)$ is tropical, then $|\Sigma|=\Trop(X^\circ)$ by \cite[Proposition~6.5]{LuxtonQu}, so $\Sigma$ induces a polyhedral structure on $\Trop(X^\circ)$. The central fiber $X_0$ of $\X$ can be written as a union of locally closed subschemes labelled by faces of $\Sigma$.
Let $X_F^\circ=\X\cap U_F$.  Let $X_F$ denote the closure of $X_F^\circ$ in $\X$.  For $w\in\Int{F}$, the initial degeneration, $\init_w X^\circ$ is invariant under the torus $T_F$.   Moreover, there is a non-canonical isomorphism
\[\init_w X^\circ\cong T_F\times X_F^\circ. 
\]
In this case, $\Trop(\init_w X^\circ)$ is invariant under translation by elements in $(N_F)_\R$.  If $F$ is a top-dimensional cell of $\Sigma$, then $m(F)$, the {\em multiplicity} of $F$ is the length of the zero-dimensional scheme $\init_w X^\circ/T_F$.  With these multiplicities $\Sigma$ is a balanced weighted rational polyhedral complex.

There is a natural notion of smoothness called sch\"onness.  

\begin{definition}\label{d:schon} A subvariety $X^\circ$ of $\T$ is {\em sch\"on}
if there exists a tropical pair $(X^\circ,\P)$ such that the multiplication map
$$m: \T \times_{\O} \X \rightarrow \P$$ is smooth.
\end{definition}

\begin{proposition} (Luxton-Qu) \cite[Theorem~6.13]{LuxtonQu} \label{p:luxtonqu} Let $X^\circ$ be sch\"on.  If $\Sigma$ is any polyhedral complex supported on $\Trop(X^\circ)$, then $(X^\circ,\P(\Sigma))$ is tropical.  Furthermore, the multiplication map $m: \T \times_{\O} \X \rightarrow \P(\Sigma)$ is smooth.
\end{proposition}

Sch\"onness has a characterization in terms of initial degenerations (ref. \cite[Lemma~2.7]{Hacking}).

\begin{proposition} \label{prop:schon} The following are equivalent:
\begin{enumerate}
\item $X^\circ$ is sch\"on.
\item $\init_w X^\circ$ is smooth for all $w \in \Trop(X^\circ)$.
\end{enumerate}
\end{proposition}

The notions of tropical pairs and sch\"on varieties originated in the work of Tevelev \cite{TevComp} and were extended to the non-constant coefficient case by Luxton-Qu \cite{LuxtonQu}.  We can speak of a subvariety $X^\circ\subset (\bk^*)^n$ being part of a tropical pair and being sch\"on.  
Let $\P=\P(\Delta)$ be a toric variety, and let $X$ be the closure of $X^\circ$ in $\P$.  $(X^\circ,\P)$ is said to be tropical if $m:T_\bk\times X\rightarrow \P$ is faithfully flat and $X$ is proper.  If $m$ is smooth, then $X^\circ$ is said to be sch\"on.

The sch\"on condition is especially important when the toric scheme $\P$ is a semi-stable degeneration of toric varieties.  One can always refine a toric scheme so that we are in this case by the following proposition:

\begin{proposition} \label{p:normalcrossings} \cite[Proposition~2.3]{HelmKatz} Let $\Sigma$ be a rational polyhedral complex in $N_\R$.  There exists an integer $d$ and a subdivision $\Sigma'$ of $d\Sigma$ such the general fiber of $\P(\Sigma')$ is a smooth toric variety and the special fiber of $\P(\Sigma')$ is a divisor with simple normal crossings.  Moreover, if the generic fiber $\P(\Sigma)_\K$ is already smooth, $\Sigma'$ can be chosen to have the same recession fan as $\Sigma$.
\end{proposition}

The above proposition can be used to extend a sch\"on variety $X^\circ$ to a scheme $\X$ with simple normal crossings degeneration.  Let $(X^\circ,\P(\Sigma))$ be a tropical pair.  Choose $\P(\Sigma')$ as above.  After making a ramified base-change $\Spec \O[t^{\frac{1}{d}}]\rightarrow\Spec \O$, we have a smooth morphism $m:\T\times_\O\X\rightarrow \P(\Sigma')$ where $\X$ is the closure of $X^\circ$ in $\P(\Sigma')$.  Consequently, the generic fiber of $\X$ is smooth and the central fiber $X_0$ is a divisor in $\X$ with simple normal crossings.  In such a case, we call $(X^\circ,\P(\Sigma'))$ a {\em normal crossings pair}.  

\begin{lemma} \label{l:schonsmooth} If $X^\circ$ is sch\"on then it is smooth.  
\end{lemma}

\begin{proof}
Choose $\P=\P(\Sigma)$ as in Proposition \ref{p:normalcrossings}.  Then $m:\T\times_\O \X\rightarrow \P$ is smooth.  Restricting to the generic fiber, we have $m:T\times X \rightarrow \P_\K$ is smooth.  Since $\P_\K$ is smooth over $\K$, $T\times X$ is smooth.  It follows that $X$ and hence $X^\circ$ is smooth.
\end{proof}

The same argument gives the following lemma:

\begin{lemma} \label{l:schonclosure} Let $X^\circ\subset (\bk^*)^n$ be a sch\"on subvariety.  Let $\P(\Delta)$ be a smooth toric variety such that $(X^\circ,\P(\Delta))$ is a tropical pair.  Then $X$, the closure of $X^\circ$ in $\P(\Delta)$, is smooth.
\end{lemma} 

\begin{lemma} \label{l:schonrestriction} If $X^\circ$ is sch\"on and $(X^\circ,\P)$ is a tropical pair, then each $X_F^\circ$ is sch\"on as a subscheme of $U_F$.
\end{lemma}

\begin{proof}
Let $V_F=\overline{U}_F$ be an orbit closure.  By base-changing $m:\T\times_\O\X\rightarrow \P$ to $V_F$, we obtain smooth $m:T_\bk\times X_F\rightarrow V_F$.  Now, the torus $T_F$ acts trivially on $V_F$ from which we obtain that
\[m:(T/T_F)_\bk\times X_F\rightarrow V_F\]
is smooth.
\end{proof}

Suppose $X^\circ$ is sch\"on.
Associated to a normal crossings 
pair $(X^\circ,\P)$ is the parameterizing complex $\Gamma_{(X^\circ,\P)}$ which has appeared in the literature a number of times (notably, in the work of Speyer \cite{SpeCurves} in the case of curves and in the work of Hacking \cite{Hacking} in the constant coefficient case).  We use the definitions introduced in Section 4 of \cite{HelmKatz}.  There, the faces of $\Gamma_{(X^\circ,\P)}$ are pairs $(F,Y)$ where $F$ is a face of $\Sigma$ and $Y$ is an irreducible component of $X_F$.  Since $(X^\circ,\P)$ is a normal crossings pair, each $X_F$ is smooth by Lemma \ref{l:schonclosure}
so irreducible components do not meet.  A face $(F',Y')$ is on the boundary of $(F,Y)$ if and only if $F'$ is a face of $F$ and $Y'$ is the unique component of $X_{F'}$ containing $Y$.  The complex $\Gamma_{(X^\circ,\P)}$ naturally maps to $\Sigma$ by the map $(F,Y)\mapsto F$.  The image of this map is $\Trop(X^\circ)$.  
The bounded faces of $\Gamma_{(X^\circ,\P)}$ form the dual complex of the semistable degeneration. 
The parameterizing complex can also be thought of as the complex associated to the stratification given by components of $X_F$ for $F\in\Sigma$.
By \cite[Proposition~4.1]{HelmKatz}, the underlying topological space of $\Gamma_{(X^\circ,\P)}$ is independent of the choice of $\P$.  When $\P$ is understood, we will denote the parameterizing complex by $\Gamma_{X^\circ}$.

For a face $F\in\Sigma$, let  $\Star_\Sigma(F)$ be the star of $F$ in $\Sigma$.  It is a fan indexed by cells $G\in \Sigma$ that contain $F$.  Pick a point $w$ in the relative interior of $F$.  The cone $\overline{G}\in\Star_\Sigma(F)$ is the set of all $v\in N_\R$ such that $w+\epsilon v\in\Sigma$ for all sufficiently small $\epsilon>0$.  We define the {\em star quotient} to be $\Delta_F=\Star_\Sigma(F)/(N_F)_\R$.  If $w\in |\Sigma|$, let $F$ be the unique face containing $w$ in its relative interior.  Then  {\em the star quotient of $\Sigma$ at $w$} is defined to be $\Star_\Sigma(F)/(N_F)_\R$.

Tropicalization of initial degenerations of $X^\circ$ can be read from $\Trop(X^\circ)$:
\begin{lemma} \cite[Proposition~2.2.3]{SpeThesis} Let $(X^\circ,\P)$ be a tropical pair.  For $w\in |\Trop(X^\circ)|$, 
\[\Trop(\init_w X^\circ)=\Star_{\Trop(X^\circ)}(F)\]
 where $F$ is the unique face of $\Sigma$ containing $w$ in its relative interior.
\end{lemma}

\section{The Tropical Motivic Nearby Fiber}\label{tmnf}

The goal of this section is to define the tropical motivic nearby fiber of a tropical variety, and show that it is independent of the choice of polyhedral structure on the tropical variety. 

We continue with the notation of Section~\ref{s:tropical}, and let $(X^\circ, \P)$ be  tropical pair. 
Let $\Sigma$ denote the corresponding polyhedral structure on $\Trop(X^\circ) \subseteq N_\R$, with recession fan $\triangle$. For each face $F$ of $\Sigma$, let $X^\circ_F$ denote the corresponding subvariety of the complex torus $U_F \cong (\C^*)^{\codim F}$ determined by $F$, and let $\tau_F$ denote the recession cone of $F$. 

The \emph{Grothendieck ring} $K_0(\Var_\C)$ of complex algebraic varieties is the free $\Z$-module generated by isomorphism classes $[V]$ of complex varieties $V$, modulo  the relation 
\begin{equation}\label{e:cutty}
[V] = [U] + [V \setminus U], 
\end{equation}
whenever $U$ is an open subvariety of $V$. Multiplication is defined by 
\[
[V] \cdot [W] = [V \times W], 
\]
and we set $\L := [\A^1]$. For example, we compute
\[
[(\C^*)^n] = [\C^*]^n = (\L - 1)^n.  
\]

\begin{definition}\label{d:tmnf}
The \define{tropical motivic nearby fiber} $\psi_\Tr  \in K_0(\Var_\C)$ of the tropical pair $(X^\circ, \P)$ is 
\[
\psi_\Tr =   \sum_{F \in \Sigma} \, [X^\circ_F] \, (1 - \L)^{\dim F - \dim \tau_F}. 
\] 
\end{definition}

\begin{remark}
Since $X_F^\circ$ may be non-reduced, we write $[X^\circ_F]$ to denote the element corresponding  to $X_F^\circ$ with its reduced structure in the Grothendieck ring. Observe that if $X^\circ$ is sch\"on, then 
$X_F^\circ$ is smooth, and hence reduced (see Proposition~\ref{prop:schon} and Remark~\ref{r:initial} below). 
\end{remark}

\begin{remark}\label{r:initial}
If $\In_w X^\circ$ denotes the initial degeneration of $X^\circ$ with respect to an element $w$ in the relative interior of $F$, then 
\[
\In_w X^\circ \cong X^\circ_F \times (\C^*)^{\dim F}. 
\]
Hence
\[
\psi_{\Tr} =  \sum_{F \in \Sigma} \frac{ (-1)^{\dim F} \, [\In_w X^\circ] }{ (1 - \L)^{\dim \tau_F}}. 
\]
\end{remark}

If $B$ is a finite poset, then
the M\"obius function $\mu_{B}: B \times B \rightarrow \mathbb{Z}$ (see Appendix~\ref{s:hdp})
satisfies the property (known as `M\"obius inversion') that for any function $h: B \rightarrow A$ to an abelian group $A$, 
\begin{equation}\label{inversion}
h(z) = \sum_{y \le z} \mu_{B}(y,z) g(y), \textrm{ where } g(y) = \sum_{x \leq y} h(x).
\end{equation}
In the lemma below, we regard the empty face of a polytope as having dimension $-1$, and write
 $\tau \in \Int(Q)$ if the relative interior of $\tau$ is contained in the relative interior of $Q$. 

\begin{lemma}\label{l:Mobius}
Let $P$ be a $d$-dimensional polytope and let $\sigma$ be a proper (possibly empty) face of $P$. If $\sigma$ is a face of a
 polyhedral decomposition $S$ of $P$, then, for any (possibly empty) face $\sigma'$ of $\sigma$,
\[
\sum_{ \substack{  \tau \in S, \tau \in \Int(P)  \\  \tau \cap \sigma = \sigma' }} 
 (-1)^{\dim \tau} =   \left\{\begin{array}{cl} (-1)^d & \text{if } \sigma' =  \sigma  \\ 0 & \text{otherwise}. \end{array}\right.
\]
\end{lemma}

\begin{proof}
Since the link of $\sigma'$ in $P$ is contractible, 
\[
\sum_{\tau \in S,  \sigma' \subsetneq \tau} (-1)^{\dim \tau - \dim \sigma' - 1} = 1, 
\]
and hence
\begin{equation}\label{e:zero}
\sum_{\tau \in S,  \sigma' \subseteq \tau} (-1)^{\dim \tau} = 0. 
\end{equation}
Let $B$ be the poset of all faces of $\sigma$, ordered by reverse inclusion, and consider the function $h: B \rightarrow \Z$ defined by 
\[
h(\sigma') = \sum_{\substack{\tau \in S \\ \tau \cap \sigma = \sigma' }} (-1)^{\dim \tau}.
\] 
With the notation of \eqref{inversion}, \eqref{e:zero} implies that $g(\sigma') = 0$ for all $\sigma' \in B$. 
Hence M\"obius inversion implies that $h(\sigma') = 0$. 

If $\tilde{B}$ denotes the poset of all (possibly empty) faces of $P$, ordered by  inclusion, 
then the  M\"obius function of $\tilde{B}$ is given by $\mu(Q', Q) = (-1)^{\dim Q - \dim Q'}$ whenever 
$Q' \subseteq Q$. 
Consider the function $h: \tilde{B} \rightarrow \Z$ defined by 
\[
h(Q) = \sum_{ \substack{  \tau \in S, \tau \in \Int(Q)  \\  \tau \cap \sigma = \sigma' }} 
 (-1)^{\dim \tau} .
\] 
With the notation of \eqref{inversion}, 
\[
g(Q) = \sum_{ \substack{  \tau \in S, \tau \subseteq Q  \\  \tau \cap \sigma = \sigma' }} 
 (-1)^{\dim \tau} .
\]
Note that the above sum is zero unless $\sigma' \subseteq Q$, in which case 
\[
g(Q) = \sum_{ \substack{  \tau \in S|_Q  \\  \tau \cap (\sigma \cap Q) = \sigma' }} 
 (-1)^{\dim \tau} .
\]
By the above discussion, the latter sum is zero unless $\sigma \cap Q = Q$, in which case $g(Q) = (-1)^{\dim \sigma'}$. M\"obius inversion then implies that 
\[
h(P) = (-1)^d \sum_{\sigma' \subseteq Q \subseteq \sigma} (-1)^{\dim Q - \dim \sigma'} 
=   \left\{\begin{array}{cl} (-1)^d & \text{if } \sigma' =  \sigma  \\ 0 & \text{otherwise}. \end{array}\right.
\]
\end{proof}

\begin{remark}
It would be interesting to have a geometric interpretation and proof of the lemma above involving Euler characteristics. 
\end{remark}

\begin{theorem}\label{t:independence}
The tropical motivic nearby fiber  $\psi_{\Tr} = \psi_{\Trr}$ of a tropical pair $(X^\circ, \P)$
is independent of the choice of polyhedral structure $\Sigma$ on $\Trop(X^\circ)$, and only depends on the corresponding recession fan $\triangle$.
\end{theorem}
\begin{proof}
Suppose $\Sigma'$ is a polyhedral structure on $\Trop(X^\circ)$ induced by a tropical pair $(X^\circ, \P')$, and  with the same recession fan $\triangle$. After taking a common refinement, we may and will assume that $\Sigma'$ is a refinement of $\Sigma$. Consider the corresponding proper morphism of toric varieties 
$\P' \rightarrow \P$. If $F'$ is a face of $\Sigma'$, then the relative interior of $F'$ lies in the relative interior $\Int(F)$ of a unique face $F$ of $\Sigma$. By standard toric geometry (see, for example, \cite{FulIntroduction}),  the corresponding morphism of tori $U_{F'} \rightarrow U_F$ factors as 
\[
U_{F'} \cong U_F \times (\C^*)^{\dim F - \dim F'} \rightarrow U_F,
\]
where the second map is projection onto the first co-ordinate. By Proposition~\ref{p:tropical}, $X^\circ_{F'}$ is the pullback of $X^\circ_F$, and hence $[X^\circ_{F'}] = [X^\circ_F](\L - 1)^{\dim F - \dim F'}$. We compute 
\begin{align*}
&\psi_{\TR} =  \sum_{F' \in \Sigma'} \, [X^\circ_{F'}] \, (1 - \L)^{\dim F' - \dim \tau_{F'}} \\
 &=  \sum_{F \in \Sigma} [X^\circ_F]  (1 - \L)^{\dim F - \dim \tau_{F}}  \sum_{\substack{F' \in \Sigma' \\ \Int(F') \subseteq \Int(F) }} (-1)^{\dim F - \dim F'} (1 - \L)^{\dim \tau_{F} - \dim \tau_{F'}}.
\end{align*}
Hence, it is enough to show that
\[
 \sum_{\substack{\Int(F') \subseteq \Int(F) \\  \dim \tau_{F'} = n}} (-1)^{\dim F - \dim F'} (1 - \L)^{\dim \tau_{F} - \dim \tau_{F'}} =  \left\{\begin{array}{cl} 1 & \text{if } n = \dim \tau_F  \\ 0 & \text{otherwise}. \end{array}\right.
\]
This follows directly from Lemma~\ref{l:Mobius} if we set $P = \sigma_F \cap H$ to be the intersection of the cone $\sigma_F$ over 
$F \times 1$ in $N_\R \times \R$ with an affine hyperplane $H$, chosen such that $P$ is a polytope not containing the origin, if $S$ equals the polyhedral decomposition of $P$ induced by $\Sigma' \cap |\sigma_F|$, and if $\sigma = \tau_F \cap H$ equals the intersection of $P$ with $N_\R \times \{ 0 \}$. 

\end{proof}

\section{The Motivic Nearby Fiber and Limit Hodge Structures}\label{s:limit}

The goal of this section is to recall some results on the motivic nearby fiber and limit mixed Hodge structure of  a degeneration of complex varieties. We refer the reader to \cite{BitMotivic}, \cite{PSHodge} and \cite{PSMixed} for details and proofs of the statements below. 

Recall from Section~\ref{tmnf} that the assignment of a complex variety $V$ to its class $[V]$ in the 
Grothendieck ring $K_0(\Var_\C)$ is the universal invariant of complex varieties satisfying the  relation \eqref{e:cutty}. We recall the following specialization of this invariant.

\excise{
 is the free $\Z$-module generated by isomorphism classes $[V]$ of complex varieties $V$, modulo  the relation 
\begin{equation}\label{e:cut}
[V] = [U] + [V \setminus U], 
\end{equation}
whenever $U$ is an open subvariety of $V$. Multiplication is defined by 
\[
[V] \cdot [W] = [V \times W], 
\]
and we set $\L := [\A^1]$. That is, the assignment $V \mapsto [V]$ is the universal invariant of complex varieties satisfying the  relation \eqref{e:cut}. We recall the following specialization of this invariant. 
}

 In  \cite{DelTheory}, Deligne proved that the  $m^{\textrm{th}}$ cohomology group  with compact supports $H^m_c  (V)$ of
every $d$-dimensional complex algebraic variety  $V$ admits a canonical mixed Hodge structure.
That is, $H^m_c  (V)$ admits  an 
 increasing filtration $W_\bullet$ called the \define{weight filtration} and  a decreasing filtration $F^\bullet$  called the  \define{Hodge filtration}, such that the Hodge filtration induces a pure Hodge structure of weight $k$ on the $k^{\textrm{th}}$ graded piece 
 of the weight filtration.  
We refer the reader to \cite{PSHodge} for an introduction to mixed Hodge structures.  
In particular, we have an induced decomposition
\[
H^m_c (V) \cong \bigoplus_{0 \le p,q \le d} H^{p,q}(H^m_c(V)), 
\] 
where $H^{q,p}(H^m_c(V; \C)) = \overline{H^{p,q}(H^m_c(V))}$. Throughout, we write 
$h^{p,q}(H^m_c(V)) = \dim  H^{p,q}(H^m_c(V))$. 
If 
\[
e^{p,q}(V) = e^{q,p}(V) = \sum_{m} (-1)^m h^{p,q}(H^m_c(V)),
\]
 then the \define{Hodge-Deligne polynomial} of $V$ is defined by
\[
E(V) = E(V; u,v) = \sum_{p,q} e^{p,q}(V) u^p v^q, 
\] 
and there is a well-defined ring homomorphism
\[
E: K_0(\Var_\C) \rightarrow \Z[u,v], \: \: E([V]) = E(V).  
\]

\begin{remark}
In fact, the ring homomorphism above factors through a specialization with values in the Grothendieck ring of complex Hodge structures $K_0( \mathfrak{hs}  )$, sending $[V]$ to $\sum_m (-1)^m H^m_c  (V)$
(see, for example, \cite[Corollary~3]{PSHodge}). 
\end{remark}

\begin{remark}\label{r:orbifold}
If $V$ is smooth and complete, 
then $H^m (V) = \oplus_{p + q = m} H^{p,q}(V)$ admits a pure Hodge structure of weight $m$, and $E(V) = \sum_{p,q} (-1)^{p + q} h^{p,q}(V) u^p v^q$. 
\end{remark}

\begin{remark}
By the above remark, $E(\P^1) = uv + 1$, and hence $E(\L) = E(\P^1) - E(\pt) = uv$. 
\end{remark}

\begin{remark}\label{r:euler}
For any variety $V$, $E(V; 1, 1) = e(V)$ equals the topological Euler characteristic of $V$. 
\end{remark}

Let $Y$ be a smooth, connected, $(d + 1)$-dimensional (complex) manifold, and let $f: Y \rightarrow \D$ be a proper map to the unit disk $\D$, which is smooth over the punctured disk $\D^* = \D \setminus \{ 0 \}$. 
One can associate to this data the \define{motivic nearby fiber} $\psi_f$ in $K_0(\Var_\C)$, 
which is invariant under base change of $\D$, and satisfies the  property that if $Z$ is smooth and connected, and $\pi:Z \rightarrow Y$ is an isomorphism away from the central fiber $Y_0 = f^{-1}(0)$, then $\psi_f = \psi_{f \circ \pi}$
\cite[Remark~2.7]{BitMotivic}. If $f$ is a \define{semi-stable degeneration}, in the sense that 
the central fiber $Y_0$ is a reduced, simple normal crossing divisor, then the motivic nearby fiber has the following description. If $\{ D_i \}_{i \in \{ 1, \ldots, n \}}$ denotes the irreducible components of $Y_0$, and, for every non-empty subset  $I \subseteq \{ 1, \ldots ,  n \}$, we set $D_I^\circ = \cap_{i \in I} D_i \setminus \cup_{j \notin I} D_j$,  then 
\begin{equation}\label{e:mnf}
\psi_f = \sum_{I \subseteq \{ 1, \ldots ,  n \}} [D_I^\circ ] (1 - \mathbb{L})^{|I| - 1}.
\end{equation}

\begin{remark}\label{r:puncture}
It follows from the discussion above that the motivic nearby fiber $\psi_f$ is determined by the smooth, proper map $f : Y' = f^{-1}(\D^*) \rightarrow \D^*$ over the punctured disk. Conversely, any smooth, proper map  $f : Y'  \rightarrow \D^*$ can be extended to a map $f: Y \rightarrow \D$, where $Y$ is a connected manifold, i.e. by extending the family  to $\D$  and resolving singularities, and hence one may consider the corresponding invariant $\psi_f$. 
\end{remark}



For a fixed $t \in \D^*$, the $m^{\textrm{th}}$ cohomology group of the smooth, complete fiber $Y_t$, which we will often denote $Y_{\gen}$, admits 
 a mixed Hodge structure called the \define{limit mixed Hodge structure}, which only depends on the family over $\D^*$, and is invariant under base change. 
 The theory of limit mixed Hodge structures was developed by many authors including Deligne, Katz, 
Clemens~\cite{CleDegeneration}, Schmid~\cite{SchVariation} and
Steenbrink~\cite{SteLimits}. 
 Throughout, we let $H^m(Y_\infty)$ denote $H^m(Y_{\gen})$ with the limit mixed Hodge structure, and 
 write
  \[
 H^m (Y_\infty) \cong \bigoplus_{
 p, q \le \min\{m,d\}} H^{p,q} (H^m (Y_\infty)). 
 \]
 The corresponding Hodge and weight filtrations may roughly be described as follows, and we refer the reader to \cite[Chapter~7]{ZolMonodromy} for details. 
 Firstly, the usual Hodge filtrations $F^p H^m(Y_t)$ on the non-zero fibers have a limit, in an appropriate sense, when $t \rightarrow 0$, called the \define{limit Hodge filtration}.   In particular, 
 $\dim F^p H^m(Y_{\gen}) = \dim F^p H^m(Y_\infty)$. That is, 
 \begin{equation}\label{e:Hodge}
h^{p,m - p}(Y_{\gen}) = \sum_q h^{p,q}(H^m(Y_\infty)),
 \end{equation}
 and hence the 
 limit Hodge numbers $\{ h^{p,q}(H^m(Y_\infty)) \}$ refine the usual Hodge numbers of a general fiber. 
  Secondly, 
 it is a classical result of Ehresmann (see \cite{MKComplex} for a proof) that $f: Y \rightarrow \D$ is a locally trivial $C^{\infty}$-fibration over $\D^*$. 
 In particular, 
 we may consider the \define{monodromy transformation}
\[
T: H^m (Y_\infty) \rightarrow H^m (Y_\infty). 
\]
A classical result of Landman~\cite{LanPicard} states that $T$ is quasi-unipotent i.e. some multiple of $T$ is unipotent. Applying base change to $\D^*$ i.e. pulling back the family $Y' \rightarrow \D^*$ via the covering map $\D^* \rightarrow \D^*$, $t \mapsto t^k$, has the effect of replacing $T$ by $T^k$. Hence we may and will assume that $T$ is unipotent. In this case, Landman further showed that $(T - I)^{m + 1} = 0$. 
By standard linear algebra, the nilpotent operator $N  = \log T  = \sum_{k \ge 1} (-1)^{k - 1} (T - I)^k/k$ induces a weight filtration on $H^m (Y_\infty)$, called the \define{monodromy weight filtration}, which determines and is determined by the Jordan block decomposition of $N$.  
Indeed, we may inductively define a unique increasing filtration 
\[
0 \subseteq W_0 \subseteq W_1 \subseteq \cdots \subseteq W_{2m} = H^m (Y_\infty),
\]
with associated graded pieces $\Gr^W_{k} := W_k/W_{k - 1}$, satisfying the following properties for any non-negative integer $k$,
\begin{enumerate}
\item $N( W_k ) \subseteq W_{k - 2}$,
\item\label{e:hard} the induced map $N^k: \Gr^W_{m + k} \rightarrow \Gr^W_{m - k}$ is an isomorphism. 
\end{enumerate}
For example,  $W_0 = \im N^m$ and $W_{2m - 1} = \ker N^m$. 
 In fact, the map $N: H^m (Y_\infty) \rightarrow H^m (Y_\infty)$ is a morphism of mixed Hodge structures of type $(-1,-1)$.
 
 \begin{remark}\label{r:monodromy}
 It follows that for $k \ge 0$, if
$H^{p,m- p} (Y_{\gen}) = 0$ for  all $p < \frac{m - k}{2}$, then $N^{k + 1} = 0$ \cite[Corollary~11.42]{PSMixed}. Indeed, by \eqref{e:Hodge} and the symmetry of the Hodge numbers of $Y_{\gen}$,
\[
H^{p,q}(H^m(Y_\infty)) = 0 \textrm{ unless }  \frac{m - k}{2} \le p \le  \frac{m + k}{2}. 
\]
Hence, either the source or the target of the induced map $N^{k + 1}: H^{p.q}(H^m(Y_\infty))  \rightarrow H^{p - k - 1.q - k - 1}(H^m(Y_\infty))$ is zero.
\end{remark}
 
The \define{limit Hodge-Deligne polynomial} is defined to be
\[
E(Y_\infty) = E(Y_\infty;u,v) =  \sum_{p,q} e^{p,q}(Y_{\infty}) u^p v^q,
\]
where 
\[
e^{p,q}(Y_{\infty})  = e^{q,p}(Y_{\infty})  = \sum_m (-1)^m h^{p,q}(H^m(Y_{\infty})). 
\]
\begin{remark}\label{r:xycharactersitic}
It follows from \eqref{e:Hodge} that 
\begin{equation}\label{e:v=1}
E(Y_\infty;u,1) = E(Y_{\gen};u,1). 
\end{equation}
In particular, by Remark~\ref{r:euler}, $E(Y_\infty;1,1)$ equals the topological Euler characteristic of 
$Y_{\gen}$. 
\end{remark}

A deep result of Steenbrink \cite[Corollary~4.20]{SteLimits} gives a spectral sequence which converges to determine the limit mixed Hodge structure. In particular, Steenbrink's result implies that the motivic nearby fiber specializes to the limit Hodge-Deligne polynomial \cite[Corollary~11.26]{PSMixed}. That is, 
under the ring homomorphism
$E: K_0(\Var_\C) \rightarrow \Z[u,v]$, 
\begin{equation}\label{e:specialize}
E(\psi_f) = E(Y_\infty;u,v). 
\end{equation}

In particular, if the central fiber $Y_0$ is a reduced, simple normal crossing divisor with irreducible components $\{ D_i \}_{i \in \{ 1, \ldots, n \}}$, then we conclude that 
\[
E(Y_\infty) = \sum_{I \subseteq \{ 1, \ldots ,  n \}} E(D_I^\circ ) (1 - uv)^{|I| - 1}.
\]

\begin{remark}\label{r:simplicial}
With the notation above, 
let $S$ be  the  simplicial complex with $k$-dimensional faces indexed by $(k + 1)$-fold intersections of the components of the central fiber $Y_0 = \sum_{i = 1}^n D_i$, with associated topological space $|S|$. As explained in \cite{MorClemens}, the weight $0$ part of $H^m(Y_\infty)$ is described by 
\[
H^{0,0}(H^m(Y_\infty)) 
\cong H^m(|S|),
\]
and $e^{0,0}(Y_\infty)$ is equal to the Euler characteristic of $|S|$. 
 \end{remark}
 
 \begin{remark}\label{r:genus}
As in \cite[Section~I.4]{CleDegeneration}, and with the notation above, 
$H^{d,0}(H^m(Y_\infty)) = 0$ for $m \ne d$, and 
\[
H^{d,0}(H^d(Y_\infty)) \cong \bigoplus_{i = 1}^n H^{d,0}(D_i).
\]
  If $Z$ is a smooth, complete $d$-dimensional variety, then the \define{genus} of $Z$ is 
  $p_g(Z) = h^{d,0}(Z) = \dim H^0(Z, \Omega_Z^d)$. We conclude that 
\[
e^{d,0}(Y_\infty)   = (-1)^d \sum_{i = 1}^n p_g(D_i).
\] 
\end{remark}

\section{Applications for sch\"on subvarieties of tori}\label{s:schon}

The goal of this section is to give a geometric interpretation of the tropical motivic nearby fiber of a sch\"on subvariety of a torus, and deduce some Hodge-theoretic and topological consequences. 

We continue with the notation of Section~\ref{s:tropical}, and let $\O$ be  the ring of germs of analytic functions in $\C$ in a neighborhood of the origin. Let $X^\circ \subseteq T$ be a sch\"on subvariety, and let $(X^\circ, \P)$ be a tropical pair with corresponding polyhedral structure $\Sigma$ on $\Trop(X^\circ)$, and unimodular recession fan $\tri$. 
Let $\mathcal{X}$ denote the closure of $X^\circ$ in $\P$, with smooth generic fiber $X$. 

As in Section~\ref{tmnf}, we consider the tropical motivic nearby fiber $\psi_{(X^\circ, \tri)}$ in the Grothendieck ring  $K_0(\Var_\C)$ of complex varieties (cf. Theorem~\ref{t:independence}).  
On the other hand,  we may fix a disk $\D$ of sufficiently small radius centered at the origin in $\C$, and consider the generic fiber $X$ as a smooth, complex manifold $X_{\an}$ admitting a smooth, proper map $f: X_{\an} \rightarrow \D^*$ to the punctured disk $\D^* = \D \setminus \{ 0 \}$. Let $\psi_f \in K_0(\Var_\C)$ be the corresponding motivic nearby fiber $\psi_f$ (cf. Remark~\ref{r:puncture}), and let 
$X_{\gen}$ denote a fixed fiber of $f$.  Observe that the invariants $\psi_{(X^\circ, \tri)}$ and
$\psi_f$, as well as $X_{\gen}$, do not depend on $\P$, but only on the associated recession fan $\tri$. 

\begin{theorem}\label{t:Schon}
With the notation above, if $X^\circ \subseteq T$ is sch\"on and $\tri$ is a unimodular recession fan associated to the tropical variety $\Trop(X^\circ)$, then  $\psi_{(X^\circ, \tri)} = \psi_f$.
\end{theorem}
\begin{proof}
By Proposition~\ref{p:normalcrossings} and its succeeding discussion, we may and will assume that $(X^\circ, \P)$ is a normal crossings pair. In particular, we may consider $\X$ as 
a smooth complex manifold $\mathcal{X}_{\an}$ admitting a proper map $f: \mathcal{X}_{\an} \rightarrow \D$, with reduced, simple normal crossings central fiber $X_0$. In this case, $\psi_f$ is computed via the formula \eqref{e:mnf}. That is, 
\[
\psi_f = \sum_{I \subseteq \{ 1, \ldots ,  n \}} [D_I^\circ ] (1 - \mathbb{L})^{|I| - 1}, 
\]
where $\{ D_i \}_{i \in \{ 1, \ldots, n \}}$ denotes the irreducible components of $X_0$, and, for every non-empty subset  $I \subseteq \{ 1, \ldots ,  n \}$, $D_I^\circ = \cap_{i \in I} D_i \setminus \cup_{j \notin I} D_j$. On the other hand, the irreducible components of $X_0$ are indexed by the vertices of $\Sigma$, and $X_0$ is the disjoint union of the locally closed subvarieties 
$\{ X_F^\circ \mid F \in \Sigma \}$. 
Moreover, $X_{F}^\circ$ is contained in precisely $\dim F - \dim \tau_{F} + 1$ of the irreducible components of $X_0$. The result now follows from Definition~\ref{d:tmnf}. 
\end{proof}

The following corollary is immediate from the discussion of Section~\ref{s:limit}. Let $E(X_\infty)$ denote the limit Hodge-Deligne polynomial associated to $f: X_{\an} \rightarrow \D^*$.

\begin{corollary}\label{c:topological}
With the notation above, if $X^\circ \subseteq T$ is sch\"on and $\tri$ is a unimodular recession fan associated to the tropical variety $\Trop(X^\circ)$, then
\[
E(X_\infty) =  \sum_{F \in \Sigma} \, E(X_{F}^\circ) \, (1 - uv)^{\dim F - \dim \tau_F}. 
\] 
In particular, 
\[
E(X_{\gen}; u, 1) =  \sum_{F \in \Sigma} \, E(X_{F}^\circ; u, 1) \, (1 - u)^{\dim F - \dim \tau_F},  
\]
and the topological Euler characteristic of $X_{\gen}$ is given by
\[
e(X_{\gen}) = \sum_{\substack{F \in \Sigma \\ \dim F = \dim \tau_F}} \, e(X_{F}^\circ). 
\]
\end{corollary}
\begin{proof}
The first statement follows from Theorem~\ref{t:Schon} and the fact that the motivic nearby fiber specializes to the limit Hodge-Deligne polynomial \eqref{e:specialize}. The second and third statements follow from Remark~\ref{r:xycharactersitic}.  
\end{proof}

Recall from Section~\ref{s:tropical} that $\Gamma_{X^\circ}$ denotes the parametrizing complex
of $X^\circ$. 
It follows from Remark~\ref{r:simplicial} that the weight $0$ part of $H^m(X_\infty)$ is described by 
\begin{equation}\label{e:weight0}
H^{0,0}(H^m(X_\infty))
\cong H^m(|\Gamma_{X^\circ}|).
\end{equation}
Recall from Section~\ref{s:limit} that we have an isomorphism 
\[
N^m: \Gr^W_{2m} H^m(X_\infty) \rightarrow \Gr^W_{0} H^m(X_\infty).
\]
 Hence, the $m^{\textrm{th}}$ Betti number $b_m(\Gamma_{X^\circ})$ of $|\Gamma_{X^\circ}|$ satisfies
\[
b_m(\Gamma_{X^\circ}) = h^{0,0}(H^m(X_\infty)) \le h^{i,i}(H^m(X_\infty)) \textrm{ for } 0 \le i \le m.  
\]
In particular, by \eqref{e:Hodge}, we immediately obtain the following corollary. 

\begin{corollary}\label{c:Betti}
With the notation above, if $X^\circ \subseteq T$ is sch\"on and $\tri$ is a unimodular recession fan associated to the tropical variety $\Trop(X^\circ)$, then the parametrizing complex $\Gamma$ satisfies 
\[
b_m(\Gamma_{X^\circ}) \le \min_{p + q = m} h^{p,q}(X_{\gen}). 
\]
\end{corollary}

This strengthens the result of Helm and the first named author \cite[Corollary~5.8]{HelmKatz} that 
\[
b_m(\Gamma_{X^\circ}) \le \frac{b_m(X_{\gen})}{m + 1}. 
\]


Recall that if $V$ is a $d$-dimensional, pure-dimensional,  smooth, complete variety, then the \emph{genus} of $V$ is $p_g(V) = h^{d, 0}(V)$.
Let $\ver(\Sigma)$ denote the set of vertices of $\Sigma$. 

\begin{corollary}\label{c:genus}
With the notation above, if $X^\circ \subseteq T$ is sch\"on and $\tri$ is a unimodular recession fan associated to the tropical variety $\Trop(X^\circ)$, and if $X_v$ is smooth for each $v \in \ver(\Sigma)$, then 
$H^{d,0}(H^m(X_\infty)) = 0$ for $m \ne d$, and 
\[
h^{d,0}(H^d(X_\infty)) = \sum_{ v \in \ver(\Sigma) } p_g(X_v).
\]
In particular, if $\Gamma_{X^\circ}$ denotes the parametrizing complex of $X^\circ$, then 
\[
b_d(\Gamma_{X^\circ}) + \sum_{ v \in \ver(\Sigma) } p_g(X_v) \le p_g(X_{\gen}). 
\]
\end{corollary}
\begin{proof}
By Theorem~\ref{t:Schon},
\[
e^{d,0}(X_\infty) = \sum_{v \in \ver(\Sigma)} e^{d,0}(X_v^\circ). 
\]
On the other hand, $e^{d,0}(X_v^\circ) = e^{d,0}(X_v) = (-1)^{d} p_g(X_v)$ by the additivity property of Hodge-Deligne polynomial, the fact that the dimension of $X_v \setminus X_v^\circ$ is strictly less than $d$, and Remark~\ref{r:orbifold}. On the other hand, by Remark~\ref{r:genus},
$H^{d,0}(H^m(X_\infty)) = 0$ for $m \ne d$, and hence
\[
e^{d,0}(X_\infty) = (-1)^d h^{d,0}(H^d(X_\infty)).
\]
The first statement is now immediate, while the second statement follows from \eqref{e:weight0}, together with the following inequality which is an immediate consequence of 
\eqref{e:Hodge},
\[
h^{0,0}(H^d(X_\infty)) + h^{d,0}(H^d(X_\infty))  \le p_g(X_{\gen}). 
\]
\end{proof}

\begin{remark}
The assumption that  $X_v$ is smooth for each $v \in \ver(\Sigma)$ in the above corollary is satisfied if 
$(X^\circ, \P)$ is a normal crossings pair (see Section~\ref{s:tropical}, cf. Remark~\ref{r:genus}).
\end{remark}


\begin{example}[Curves]
Consider the case when $X^\circ$ is $1$-dimensional, and let 
$C \rightarrow \D^*$ denote the corresponding  smooth family of curves with non-zero fiber $C_{\gen}$, and parametrizing complex $\Gamma$. After possibly scaling and refining $\Trop(X^\circ)$, we can assume that $X_{\an}$ is a smooth manifold.
Moreover, $C_v$ is smooth for each $v \in \ver(\Sigma)$.
Then the zeroth and top cohomology groups are trivial 
\[
b_0(C_{\gen}) = b_2(C_{\gen}) = 1.
\]
The limit mixed Hodge numbers of the middle cohomology are given by
\[
h^{0,0}(H^1(C_\infty)) =  h^{1,1}(H^1(C_\infty)) = b_1(\Gamma), 
\]
\[
h^{1,0}(H^1(C_\infty)) = h^{0,1}(H^1(C_\infty)) = \sum_{ v \in \ver(\Sigma) } g(C_v).
\]
In particular, the genus of $C_{\gen}$ is given by the formula
\[
g(C_{\gen}) = b_1(\Gamma_C) + \sum_{ v \in \ver(\Sigma) } g(C_v).
\]
\end{example}

\section{Limit Hodge structures of hypersurfaces}\label{s:hypersurfaces}

The goal of this section is to explicitly compute the limit mixed Hodge structure of a sch\"on family of hypersurfaces. We continue with the notation of Section~\ref{s:tropical}, and refer the reader to \cite{FulIntroduction} for basic facts on toric varieties. 

As in Section~\ref{s:schon},  let $\O$ be  the ring of germs of analytic functions in $\C$ in a neighborhood of the origin.  Let $X^\circ \subseteq T = \Spec \K[M]$ be a sch\"on hypersurface. 
That is, $X^\circ = \{ \sum_{u \in M} \alpha_u x^u = 0 \} \subseteq T$, for some $\alpha_u \in \K$. 
The \define{Newton polytope} $P$ of $X^\circ$ is the convex hull of $\{ u \in M \mid \alpha_u \ne 0 \}$, and the function 
$P \cap M \rightarrow \Z$, $u \mapsto \ord \alpha_u$,
induces a regular, lattice polyhedral decomposition $\T$ of $P$. Explicitly, the faces of $\T$ are the projections 
of the bounded faces of the convex hull of $\{ (u, \lambda) \mid \alpha_u \ne 0, \lambda \ge \ord \alpha_u \}$ in $M_\R \times \R$.  The tropicalization $\Trop(X^\circ)$ is `dual' to the polyhedral decomposition  $\T$, in the sense that $\Trop(X^\circ)$ has a natural polyhedral structure such that its  faces are in bijection with the positive dimensional faces of $\T$ (see, for example, \cite[Section~3]{RSTFirst}).  The corresponding recession fan of $\Trop(X^\circ)$ 
is the fan obtained from the normal fan of $P$ by removing its maximal cones. 
We will assume throughout this section that $P$ is $(d + 1)$-dimensional and that $P$ is \define{almost smooth} in the sense that 
the recession fan of  $\Trop(X^\circ)$ is unimodular. We remark that if $P$ is not almost smooth, then one can apply Proposition~\ref{p:normalcrossings} to obtain a unimodular refinement of the recession fan. 

As in Section~\ref{s:schon}, we may consider the corresponding proper map $f: X_{\an} \rightarrow \D$ to a disk $\D$ of sufficiently small radius, which is smooth over the punctured disk $\D^*$. A fixed non-zero fiber $X_{\gen}$
is a smooth hypersurface of the smooth toric variety $Y_P$ associated to the recession fan, and is sch\"on with respect to the associated complex torus. The central fiber $X_0$ is a disjoint union of locally closed subvarieties $\{ X_Q^\circ \mid Q \in \T \}$, each of which is Sch\"on with respect to its corresponding complex torus $U_Q \cong (\C^*)^{\dim Q}$.    If $\T^{(i)}$ denotes the set of faces of $\T$ whose relative interior lies in the relative interior of a codimension $i$ face of $P$, then Corollary~\ref{c:topological} implies that 
\begin{equation}\label{e:formula}
E(X_\infty;u,v) = \sum_{i =  0}^{d} \sum_{Q \in \T^{(i)}} E(X_{Q}^\circ)(1 - uv)^{\codim Q - i}. 
\end{equation}

The notion of a  sch\"on or \define{non-degenerate} hypersurface $Z$ of a torus  $(\C^*)^n$ was introduced by Khovanski{\u\i} in \cite{KhoNewton}. Danilov and  Khovanski{\u\i} gave an explicit combinatorial algorithm to compute the Hodge-Deligne polynomial of a non-degenerate hypersurface of a complex torus \cite{DKAlgorithm}. Later, Batyrev and Borisov produced a combinatorial formula \cite[Theorem~3.24]{BBMirror}, which was simplified by Borisov and Mavlyutov in \cite[Proposition~5.5]{BMString}.
The formula is determined by the Ehrhart polynomial of the Newton polytope and all its faces (see, for example, \cite{BRComputing}), together with the face poset of the Newton polytope, and is stated explicitly in Theorem~\ref{TheoremA}.
In particular, together with \eqref{e:formula}, we obtain a combinatorial formula for $E(X_\infty;u,v)$. 

By Remark~\ref{r:orbifold},  the Hodge numbers of $X_{\gen}$ are also determined by this formula, since the intersection of $X_{\gen}$ with any toric stratum of $Y_P$ may be regarded as a non-degenerate hypersurface with Newton polytope a face of $P$. In particular, the topological Euler characteristic of $X_{\gen}$ is given by the formula \cite[Remark~4.5]{DKAlgorithm}
\begin{equation}\label{e:volume}
e(X_{\gen}) = \sum_{Q \subseteq P} (-1)^{\dim Q - 1} (\dim Q)! \vol(Q),
\end{equation}
where $\vol(Q)$ denotes the Euclidean volume of a face $Q$ of $P$ with respect to its affine span.

The inclusion $X_{\gen} \hookrightarrow Y_P$ induces a map on cohomology 
\[
H^m(Y_P) \rightarrow H^m(X_{\gen}), 
\] 
which, by the Lefschetz hyperplane theorem, is an isomorphism for $m < d$, and injective for $m = d$. 
Moreover,  Poincar\'e duality implies that $H^m(X_{\gen}) \cong H^{2d - m}(X_{\gen})$. 
Recall that the \define{h-vector} $\{ h_{P,k} \}$ of $P$ 
 is defined as the coefficients of the polynomial 
\[
h_P(t) = \sum_{k = 0}^{d + 1} h_{P,k} t^k =  \sum_{Q \subseteq P} t^{d + 1 - \dim Q}(1 - t)^{\dim Q},
\]
where the sum runs over all non-empty faces $Q$ of $P$. 
The projective toric variety $Y_P$ has no odd cohomology, and  $\dim H^{2m}(Y_P) = 
h^{m,m}(Y_P) = h_{P,m}$. By Remark~\ref{r:monodromy}, we deduce that, for $m \ne d$,
the monodromy operator on $\dim H^{m}(X_{\gen})$ is trivial, and hence the limit mixed Hodge structure coincides with the usual pure Hodge structure. When $m = d$, the same holds for the image of  $H^{d}(Y_P)$ in $H^{d}(X_{\gen})$.  Hence, to describe the limit mixed Hodge structure of $X$, it remains to describe the limit mixed Hodge structure on 
\[
H^{d}_{\prim}(X_\infty)  :=  \coker[H^{d}(Y_P) \rightarrow H^{d}(X_{\gen})]. 
\]
It follows that $E(X_\infty;u,v)$ determines and is determined by the limit Hodge numbers $h^{p,q}(H^m(X_\infty))$. That is, we have shown the following. 

\begin{corollary}\label{c:formula}
With the notation above, there exist explicit combinatorial formulas for the limit Hodge numbers $h^{p,q}(H^m(X_\infty))$.
\end{corollary}

\begin{remark}
In general, knowing the limit Hodge numbers $h^{p,q}(H^m(X_\infty))$ of a degeneration is strictly more information than knowing the limit Hodge-Deligne polynomial $E(X_\infty)$. 
\end{remark}

We will now present simpler combinatorial formulas for some of the limit Hodge numbers  $h^{p,q}(H^{d}(X_\infty))$. 
Below, $\Int(Q)$ denotes the relative interior of a face $Q$ of $P$. 

\begin{example}\label{e:00}
Recall that  $\Trop(X^\circ)$ is dual to the polyhedral decomposition $\T$, such that an interior edge of $\T$ with $s$ interior lattice points corresponds to a maximal, unbounded face of $\Trop(X^\circ)$ of multiplicity $s + 1$. 
It follows that 
the parametrizing complex $\Gamma_{X^\circ}$ is homotopic to a wedge of 
$d$-spheres, which can be indexed by the lattice points in the interior of $P$ which lie on either an edge or a vertex of $\T$.  
By \eqref{e:weight0},  
\[
h^{0,0}(H^{d}(X_\infty))  = b_{d}(\Gamma_{X^\circ})  = \sum_{ \substack{ Q \in \T^{(0)} \\   \dim Q \le 1    }  } \#(\Int(Q) \cap M),  
\]
and $h^{0,0}(H^m(X_\infty)) =  b_m(\Gamma_{X^\circ})  = 0$ for $0 < m < d$. 
\end{example}

We will need the following result of Danilov and  Khovanski{\u\i}. 

\begin{proposition}\cite[Proposition~5.8]{DKAlgorithm}\label{p:p0}
If $X_Q^\circ$ is a non-degenerate hypersurface of a complex torus with Newton polytope $Q$ in a lattice $M$, and $p > 0$, 
then 
\[
e^{p,0}(X_Q^\circ) = (-1)^{\dim Q + 1} \sum_{\substack{ Q' \subseteq Q  \\ \dim Q' = p + 1   }} \#(\Int(Q') \cap M).
\]
\end{proposition}

We deduce the following corollary.

\begin{corollary}\label{c:p0}
With the notation above, for $p > 0$, 
\[
h^{p,0} (H^{d}(X_\infty)) =  \sum_{  \substack{  Q \in \T^{(0)} \\ \dim Q = p + 1}}  \# (\Int(Q) \cap M). 
\]

\end{corollary}
\begin{proof}
By the above discussion, $h^{p,0}(H^m(X_\infty)) = 0$ for $m \ne d$, and hence
\[
h^{p,0} (H^{d}(X_\infty)) = (-1)^{d}e^{p,0}(X_\infty).
\]
By \eqref{e:formula} and Proposition~\ref{p:p0},
\begin{align*}
e^{p,0}(X_\infty) &= \sum_{Q \in \T} e^{p,0} (X_{Q}^\circ) \\
&= \sum_{Q \in \T} (-1)^{\dim Q + 1} \sum_{\substack{ Q' \subseteq Q  \\ \dim Q' = p + 1   }} \#(\Int(Q') \cap M) \\
&= \sum_{\substack{ Q' \in \T  \\ \dim Q' = p + 1   }} \#(\Int(Q') \cap M) \cdot \sum_{\substack{ Q \in \T \\ Q' \subseteq Q }} (-1)^{\dim Q + 1}. 
\end{align*}
Since the link of a face $Q'$ in $\T$ is homotopic to a sphere if $Q' \in \mathcal{T}^{(0)}$, and 
is contractible otherwise, it follows that 
\[
\sum_{\substack{ Q \in \T \\ Q' \subseteq Q }} (-1)^{\dim Q + 1} = \left\{ \begin{array}{c l}  (-1)^{d - 1} & \textrm{ if } Q' \in \mathcal{T}^{(0)} \\ 0 & \textrm{otherwise}. \end{array} \right. 
\]
\end{proof}

\begin{remark}
From Example~\ref{e:00}, Corollary~\ref{c:p0}, and \eqref{e:Hodge}, we recover the well-known fact that  the genus $h^{d,0}(X_{\gen})$ of $X_{\gen}$ equals the number of interior lattice points of $P$.  
\end{remark}

\begin{example}[cf. Section~\ref{s:smooth}]
Suppose that $\mathcal{T}$ is a unimodular triangulation of $P$. That is, suppose that each maximal face of $\mathcal{T}$ is isomorphic to the standard $(d  + 1)$-dimensional simplex. For each face $Q$ in 
$\T$, $X_Q^\circ$ is isomorphic to the intersection of a general linear hyperplane in projective space $\P^{\dim Q}$ with the maximal torus $(\C^*)^{\dim Q}$, and hence $E(X_F^\circ)$ is a polynomial in $uv$. 
It follows from \eqref{e:formula} and the discussion above that all non-zero limit Hodge numbers are of type $(p,p)$. Hence, if one sets $m = d$ in \eqref{e:Hodge}, then at most one term on the right hand side of the equation is non-zero, and
\[
h^{p,p}(H^{d}(X_\infty)) = h^{p, d - p}(X_{\gen}). 
\]
\end{example}

\begin{example}[Curves]
Let us assume that $P$ is a $2$-dimensional lattice polytope. By Example~\ref{e:00} and Corollary~\ref{c:p0}, 
\[
h^{0,0}(H^{1}(C_\infty))  = h^{1,1}(H^{1}(C_\infty))  =  \sum_{ \substack{ Q \in \T^{(0)} \\   \dim Q \le 1    }  } \#(\Int(Q) \cap M),  
\]
\[
h^{1,0}(H^{1}(C_\infty))  = h^{0,1}(H^{1}(C_\infty))  =  \sum_{ \substack{ Q \in \T^{(0)} \\   \dim Q = 2    }  } \#(\Int(Q) \cap M).
\]
\end{example}

\begin{example}[Surfaces]
Let us assume that $P$ is a $3$-dimensional lattice polytope.  By Example~\ref{e:00} and Corollary~\ref{c:p0}, 
\[
h^{0,0}(H^{2}(X_\infty))  = h^{2,2}(H^{2}(X_\infty))  =  \sum_{ \substack{ Q \in \T^{(0)} \\   \dim Q \le 1    }  } \#(\Int(Q) \cap M),  
\]
\begin{align*}
h^{1,0}(H^{2}(X_\infty))  = h^{0,1}(H^{2}(X_\infty)) = h^{2,1}(H^{2}(X_\infty)) &= h^{1,2}(H^{2}(X_\infty))  \\ =  \sum_{ \substack{ Q \in \T^{(0)} \\   \dim Q =  2    }  } \#(\Int(Q) \cap M),  
\end{align*}
\[
h^{2,0}(H^{2}(X_\infty))  = h^{0,2}(H^{2}(X_\infty))  =  \sum_{ \substack{ Q \in \T \\   \dim Q = 3    }  } \#(\Int(Q) \cap M).
\]
Lastly, $h^{1,1}(H^{2}(X_\infty))$ can easily be deduced from the combinatorial formula \eqref{e:volume}
for the topological Euler characteristic $e(X_{\gen})$. Explicitly, $e(X_{\gen}) = b_2(X_{\gen}) + 2$, and 
$b_2(X_{\gen}) = \sum_{p,q} h^{p,q}(H^{2}(X_\infty))$ by \eqref{e:Hodge}. 
\end{example}

\section{The Tropical Motivic Nearby Fiber of Matroidal Tropical Varieties}\label{s:smooth}

A matroidal tropical variety is one that is locally described by the matroid fans of Ardila and Klivans \cite{AKBergman}.  They were first introduced by Mikhalkin in \cite{MikhRat}  who originally called them ``smooth tropical varieties'' but has since renamed them  ``effective tropical cycles of multiplicity $1$'' in \cite{MikhRTG}.  The name of such tropical varieties will eventually be standardized in the literature but we use the matter-of-fact adjective matroidal for the time-being.  Some evidence for Mikhalkin's original name ``smooth tropical variety'' is provided by Proposition \ref{p:smooth}.

Matroids are abstract objects that axiomatize the combinatorics of linear independence.  A rank $d+1$ \define{matroid} $\M$ on a finite set $E=\{0,1,\dots,n\}$ is given by a rank function $r:2^E\rightarrow\N\cup\{0\}$ satisfying
\begin{enumerate}
\item $r(I)\leq |I|$,
\item $I\subseteq J$ implies $r(I)\leq r(J)$,
\item $r(I\cup J)+r(I\cap J)\leq r(I)+r(J)$,
\item $r(E)=d+1$.
\end{enumerate}
A \define{flat} of $\M$ is a subset $I\subseteq E$ such that for all $j\in E\setminus I$, $r(I\cup\{j\})>r(I)$. 
They form a lattice under the partial order of inclusion.
One may associate such a matroid to a $d$-dimensional linear subspace $X$ of $\P^n$.  For $I\subset E$, let $H_I$ be the coordinate subspace given by $\cap_{i\in I} \{X_i=0\}$, and set $r(I)=d-\dim(X\cap H_I)$ (where we use the convention that the dimension of the empty set is $-1$).  $I$ is a flat of $\M$ if and only if the linear space $X\cap H_I$ is not 
equal to 
 $X\cap H_J$ for any $J\supset I$.  Equivalently, if we define $H_I^*$ to be the subset of $H_I$ given by $X_j\neq 0$ for $j\not\in I$, $I$ is  a flat of $\M$ if and only if $X$ intersects $H_I^*$.

The matroid $\M$ can be encoded in a simplicial fan called the {\em matroid fan} $\Delta_\M$. Let $N$ be the lattice
\[
N = \Z^E / \< e_0 + \cdots + e_n \>.
\]
$\Delta_\M$ will be a fan in $N_\R$.  For a subset $I \subset E$, let $e_I$ be the vector 
\[
e_I = \sum_{i \in I} e_i
\]
in $N_\R$.
The rays of the matroid fan $\Delta_\M$ are $\rho_L=\R_+ e_L$ for proper flats $L \subsetneq E$ of the matroid.  More generally, the $k$-dimensional cones of the matroid fan correspond to the $k$-step flags of proper flats:  if $\cF$ is a flag of flats $L_1 \subset \cdots \subset L_k$ then the corresponding cone is $\sigma_\cF=\Span_+(\{ e_{L_1}, \ldots, e_{L_k} \})$.  Because every flag of flats in a matroid can be extended to a full flag, the matroid fan $\Delta_\M$ is of pure dimension $d$.  Each top-dimensional cone of $\Delta_\M$ is given multiplicity $1$.  Ardila and Klivans introduced this fan as the fine subdivision of the Bergman fan of a matroid.  

The tropicalization of a hyperplane arrangement complement $X^\circ$ defined over $\C$ is a fan of the form $\Delta_\M$.  In fact, let $X\subset\P^n$ be a $d$-dimensional linear space and set $X^\circ=X\cap (\C^*)^n$.  Let $\M$ be the matroid on $E=\{0,1,\dots,n\}$ associated to $X$.   Then, $\Trop(X^\circ)=\Delta_\M$.

\begin{example}
Let $\M$ be a rank $2$ matroid on $E=\{0,1,\dots,n\}$.  Let $I_1,\dots,I_s$ be the set of of rank $1$ flats.  $\Delta_{\M}$ is the union of the rays $\{ \rho_{I_k} \}_{1 \le k \le s}$.
In the special case that every element of $I$ is a rank $1$ flat, $\Delta_{\M}$ is the union of $\rho_0,\dots,\rho_n$ which is the $1$-skeleton of the fan corresponding to $\P^n$.  This is the tropicalization of a generic line in $\P^n$ which intersects each coordinate hyperplane in a generic point.
\end{example}

Moreover, if a matroid fan is the tropicalization of a variety, the variety must be a hyperplane arrangement complement.

\begin{proposition}\label{p:duck}\cite[Proposition~4.2]{KPRealization} Let $X^\circ\subset(\C^*)^n$ be a subvariety.  Suppose that $|\Trop(X^\circ)|=|\Delta_\M|$ and each top-dimensional cell of $\Trop(X^\circ)$ has multiplicity $1$.  Then the closure of $X^\circ$ in $\P _{\C}^n$ is a $d$-dimensional linear subspace whose matroid is $\M$.
\end{proposition}

The matroid $\M$ can be recovered from the underlying set $|\Delta_\M|$.  For $I\subset E$, let $\tau_I=\Span_+(\{e_i|i\in I\})$.  The set of all $\tau_I$ form the fan $\Delta$ corresponding to the toric variety $\P^n$.  Note that each open cone of $\Delta_\M$ is contained in  a unique open cone in $\Delta$.  In fact, $\sigma_\cF\subset \tau_{L_k}$. 

\begin{lemma} Let $I\subset E$.   $I$ is a proper flat of $\M$ if and only if $|\Delta_\M|\cap\tau^\circ_I\neq\emptyset$.  In that case, $r(I)$ is equal to the dimension of the tangent space of a smooth point of $|\Delta_\M|\cap\tau^\circ_I$. 
\end{lemma}

\begin{proof}
Suppose $I$ is a flat of $\M$, then $\rho_I$ is a ray in $\Delta_M$ and is contained in $\tau_I^\circ$.

Now suppose $\Delta_\M\cap \tau^\circ_I\neq \emptyset$.  Then there is an open cone $\sigma_\cF$ contained in the relative interior of  $\tau_I$.  If $\cF=\{L_1\subset \dots\subset L_k\}$ then $L_k=I$ and $I$ is a flat of $\M$.  Now, $\sigma^\circ_\cF$ is a maximal cone of $\Delta_\M\cap \tau^\circ_I$ if and only if $\cF$ is a saturated flag.  Consequently, $\dim(\sigma_\cF)=r(I)$.  In that case, the linear span of $\sigma_\cF$ is equal to the tangent space of a point in its relative interior.
\end{proof}

Let $P_\M$ be the poset of flats of $\M$ under inclusion where $\hat{0}=\emptyset$ is the unique minimal element and $\hat{1}=\{0,1,\dots,n\}$.  Then by Theorem 1 of \cite{AKBergman}, $ \Delta_\M\cap S^{n-1}$ is a geometric realization of the order complex of the poset $\Delta(P_\M\setminus\{\hat{0},\hat{1}\})$.   Consequently, $\mu(\hat{0},\hat{1})=\tilde{\chi}(|\Delta_\M|\cap S^{n-1})$ by  Proposition 3.8.6 of \cite{Stanley1}.
One may recover the M\"{o}bius function of $P_\M$ from $\Delta_M$.   

\begin{lemma} \label{l:mobiuslink} Let $I$ be a flat of $\M$ of positive rank.  Let $K$ be the set of cones of $\Delta_M$ that intersect $\tau^\circ_I$ in $\rho_I$ ordered under inclusion.  Let $\sigma_I'$ be a maximal cone in $K$.  Then 
\[\mu(\hat{0},I)=\tilde{\chi}(\lk_{\sigma_I'}(\Delta_\M)).\]
\end{lemma}

\begin{proof}
Each cone in $\Delta_M$ that intersects $\tau_I^\circ$ in $\rho_I$ is of the form $\sigma_I'=\sigma_\cF$ for $\cF$ a flag of flats $\{L_1=I\subset\dots\subset L_k\}$.  For the cone $\sigma_\cF$ to be maximal of that type, it must correspond to a saturated flag of flats. 
A cone $\sigma_{\cF'}$ contains $\sigma_\cF$ in its closure if and only if $\cF$ is the terminal segment of $\cF'$.  It  follows that $
\lk_{\sigma_I'}(\Delta_\M)$ is a geometric realization of the 
order complex $\Delta((P_\M)_I\setminus \{\hat{0},I\})$, where $(P_\M)_I$ is the principal order ideal $(P_\M)_I=\{J|J\subseteq I\}$.
The conclusion follows form Proposition 3.8.6 of \cite{Stanley1}.
\end{proof}

We can use the above lemma in combination with $\mu(\hat{0},\hat{0})=1$ to determine all values of $\mu(\hat{0},I)$.
Note that if $I$ is a rank $1$ flat, then $\mu(\hat{0},I)=\tilde{\chi}(\emptyset)=-1$.
The M\"{o}bius function can be read from the geometry of $|\Delta_\M|$ once we have chosen $\sigma_I'$ since $\lk_{\sigma_I'}(\Delta_\M)=\pi_I'(\Star_{\Delta_\M}(\sigma_I'))\cap S(N/N_{\sigma_I'})$, where $\pi_I':N\rightarrow N/N_{\sigma_I'}$ and $S(N/N_{\sigma_I'})$ is the unit sphere in $N/N_{\sigma_I'}$.

Recall that the characteristic polynomial of $\M$ is given by
\[\chi_{_\M}(q)=\sum_{I\in P_\M} \mu(\hat{0},I)q^{d-r(I)}.\]
Consequently,
\[\chi_{_\M}(q)=1-\sum_{
r(I)=1} q+
\sum_{
r(I)\geq 2} \tilde{\chi}(\lk_{\sigma'_I}(\Delta_\M))q^{d-\dim(\sigma'_I)}.\]

We have to make use of the projective motivic version of Theorem 5.2 of \cite{OSHyperplanes}.

\begin{lemma} \label{l:hyper} The class of $[X^\circ]$ in $K_0(\Var_\C)$ is given by
\[[X^\circ]=\frac{\L\chi_\M(\L)-\chi_\M(1)}{\L-1}.\]
\end{lemma}

\begin{proof}
From $[\A^n] = \L^n$, we have
\[E(\P^n)=1+\L+\dots+\L^n=\frac{\L^{n+1}-1}{\L-1}.\]
Since $X^\circ$ is expressed motivically as $\sum_{I\in P_M} \mu(\hat{0},I) [X\cap H_I]$, 
\[(\L-1)[X^\circ]=\sum_{I\in P_\M} \mu(\hat{0},I)((\L)^{d-r(I)+1}-1)=\L\chi_\M(\L)-\chi_\M(1).\]
\end{proof}

\begin{example}
Let us consider a rank $2$ matroid on $E=\{0,1,\dots,n\}$.  Let $I_1,\dots,I_s$ be its rank $1$ flats.  
By Lemma \ref{l:mobiuslink}, we know
\[\mu(\hat{0},\hat{0})=1,\ \mu(\hat{0},I_k)=-1.\]
Consequently,
\[\chi_\M(q)=q-s.\]
It follows that
\[[X^\circ]=\frac{\L(\L-s)-(1-s)}{\L-1}=\L-s+1\]

Alternatively, we know that $\Delta_\M$ is the tropicalization of the complement of $s$ points in $\P^1$.  It follows that $[X^\circ]=\L-s+1$.
\end{example}

We now globalize the notion to tropical varieties that locally look like matroid fans. 

\begin{definition} A tropicalization $\Trop(X^\circ)$ is said to be {\em matroidal with respect to a rational polyhedral  structure $\Sigma$ on $\Trop(X^\circ)$} if
\begin{enumerate}
\item Every top-dimensional cell of $\Trop(X^\circ)$ has multiplicity $1$.

\item For every face $F$ of $\Trop(X^\circ)$, the star quotient $\Star_{\Trop(X^\circ)}(F)/N_F$ has the same underlying set as some matroid fan $\Delta_{\M_F}$ for some choice of integral basis $e_1,\dots,e_l$ of $N/N_F$.
\end{enumerate}
\end{definition}

\begin{definition} A tropicalization $\Trop(X^\circ)$ is said to be {\em matroidal} if it is matroidal with respect to some $\Sigma$.
\end{definition}

Note that we have $e_0=-e_1-\dots-e_l$ in $N/N_F$.  We do not require the choice of $e_0,\dots,e_l$ to be global so $\M_F$ is not necessarily uniquely defined.  Nor do we require that there is a polyhedral structure on $\Trop(X^\circ)$ inducing the matroid fan structure on the stars of cells.   This is done by Allermann \cite{ASmooth} in developing intersection theory on matroidal tropical varieties.  His definition is also different form ours because his tropical varieties are locally modeled on the uniform matroid.  Because our results are independent of the polyhedral structure on $\Trop(X^\circ)$ we do not require this.  However, to get finer data about $X$ from $\Trop(X^\circ)$, doing so may be necessary.  Examples of matroidal tropical varieties include planar trivalent curves all of whose vertices have multiplicity $1$ (in the sense of \cite[Definition~2.16]{MikhEnum}) and tropical hypersurfaces corresponding to a Newton subdivision all of whose cells are unimodular simplices.

\begin{lemma} Suppose $\Trop(X^\circ)$ is matroidal with respect to $\Sigma$.  If $\Sigma'$ is a refinement of $\Sigma$, then $\Trop(X^\circ)$ is matroidal with respect to $\Sigma'$.
\end{lemma}

\begin{proof}
Let $F'$ be a cell in $\Sigma'$ whose relative interior is contained in the relative interior of a cell $F$ of $\Sigma$.  Then, 
$\Star_{\Trop(X^\circ)}(F')=\Star_{\Trop(X^\circ)}(F)$ and
\[\Star_{\Trop(X^\circ)}(F)/(N_F)_\R=\left(\Star_{\Trop(X^\circ)}(F')/(N_{F'})_\R\right)/\left((N_{F}/N_F')_\R\right).\]
Suppose that $\M$ is the matroid fan of $\{0,\dots,n-\dim(F)\}$ associated to the above fan with respect to a basis $e_1,\dots,e_{n-\dim(F)}$.
Consider the short exact sequence of lattices
\[0\rightarrow N_F/N_{F'}\rightarrow N/N_{F'} \rightarrow N/N_F \rightarrow 0.\]
Pick a splitting $j:N/N_F\rightarrow N/N_{F'}$ and let $e'_1,\dots,e'_{n-\dim(F)}$ be the image of the basis under this splitting.  Set $l=\dim(F)-\dim(F')$, and let $f_1,\dots,f_l$ be a basis for $N_F/N_{F'}$.  
Let $G_1,\dots,G_l$ be copies of the unique rank $1$ matroid on $1$ element, and
define $\M'$ be the matroid on $n-\dim(F')+1$ elements given by
\[\M'=\M \sqcup G_1 \sqcup \dots \sqcup G_l.\]
It is straightforward to verify that the underlying set of $\Delta_{\M'}$ with respect to the basis $\{e'_1,\dots,e'_{n-\dim(F)},f_1,\dots,f_l\}$ is 
\[j_\R(|\Delta_\M|)+(N_F/N_{F'})_\R=\Star_{\Trop(X^\circ)}(F')/(N_{F'})_\R.\]
\end{proof}

\begin{proposition} \label{p:smooth} If $\Trop(X^\circ)$ is a matroidal tropical variety, then $X^\circ$ is sch\"{o}n.  Consequently, $X^\circ$ is smooth.
\end{proposition}

\begin{proof}
We first show that every initial degeneration of $X^\circ$ is smooth.   If $w\in\Trop(X^\circ)$ is the relative interior of a cell $F$ of $\Sigma$, $\Star_{\Trop(X^\circ)}(F)/(N_F)_\R$ is a matroid fan. Now $\init_w X^\circ$ is $T_F$-invariant and we have
\[|\Trop(\init_w X^\circ/T_F)|=|\Trop(\init_w X^\circ)|/(N_F)_\R=\Star_{\Trop(X^\circ)}(F)/(N_F)_\R.\] 
By Proposition~\ref{p:duck}, $(\init_w X^\circ)/T_F$ is a hyperplane complement and $\init_w X^\circ$ is smooth.  
It follows from Proposition \ref{prop:schon} that $X$ is sch\"{o}n.  Consequently, by Lemma \ref{l:schonsmooth}, $X^\circ$ is smooth.
\end{proof}

Matroidal tropical varieties may therefore be thought of as having a very strong form of classical smoothness.  In general, even if $X^\circ$ is smooth there may be singular ${X'}^\circ$ with $\Trop(X^\circ)=\Trop({X'}^\circ)$.  A matroidal tropical variety, however, has no singular lift.  One could think in analogy with tropical general position of points: every classical lift of a set of tropical points in general position is in general position. 

Since $X^\circ$ is sch\"on, by Proposition \ref{p:luxtonqu}, $(X,\P(\Sigma))$ is a tropical pair for any choice of rational polyhedral structure $\Sigma$ on $\Trop(X^\circ)$.
For matroidal tropical varieties, there is no difference between the parameterizing complex $\Gamma_{X^\circ}$ and $\Trop(X^\circ)$:

\begin{lemma} Let $\Trop(X^\circ)$ be a matroidal tropical variety.  Then the natural parameterizing map $p:\Gamma_{X^\circ}\rightarrow\Trop(X^\circ)$ is a homeomorphism.
\end{lemma}

\begin{proof}
Let $\P$  be a toric scheme such that $(X^\circ,\P)$ is a normal crossings pair.  
 Put the polyhedral structure on $\Trop(X^\circ)$ induced from $\Sigma$.  Each cell of $\Gamma_{(X^\circ,\P)}$ is of the form $(F,Y)$ where $F$ is a cell of $\Trop(X^\circ)$ and $Y$ is a component of $X_F^\circ$.  Since each $X_F^\circ$ is a linear space, it follows that there is only one cell of $\Gamma_{(X^\circ,\P)}$ lying above each cell of $\Trop(X^\circ)$.  Consequently, $p$ is an isomorphism of polyhedral complexes, and  hence a homeomorphism.
\end{proof}

Now, $\chi_{_{\M_F}}(q)$ can reconstructed from the geometry of $|\Trop(X^\circ)|$ by looking at the star quotient of $F$ and applying the methods above.

By Lemma \ref{l:hyper}, we have the following:

\begin{corollary}\label{c:sformula} Let $\Sigma$ be a polyhedral structure on $\Trop(X^\circ)$ with recession fan $\Delta$, and suppose that
$\Trop(X^\circ)$ is a matroidal tropical variety.
  For $F$ a cell in $\Trop(X^\circ)$, let $\M_F$ be the matroid 
   associated to its star quotient.   The tropical motivic nearby fiber of $X_\infty$ is given by 
\[\psi_{\Trr}=-\sum_{F\in\Sigma} (\L\chi_{_{\M_F}}(\L)-\chi_{_{\M_F}}(1))(1-\L)^{\dim(F)-\dim(\tau_F)-1}\]
\end{corollary}

\begin{corollary}\label{c:sEuler} With the notation as above, the Euler characteristic of the generic fiber $X_{\gen}$ is given by
\[e(X_{\gen})= \sum_{\substack{F\in\Sigma\\ \dim(F)=\dim(\tau_F)}} (\chi_{_{\M_F}}(1)+\chi'_{_{\M_F}}(1)),\]
where $\chi'_{_{\M_F}}(q)$ denotes the derivative of  $\chi_{_{\M_F}}(q)$.
\end{corollary}

\begin{proof}
By Corollary \ref{c:topological},
\[e(X_{\gen}) = \sum_{\substack{F \in \Sigma \\ \dim F = \dim \tau_F}} \, e(X_{F}^\circ). \]
Now, $e(X_F^\circ)$ is determined by specializing $\L=1$ in
\[[X_F^\circ]=\frac{\L\chi_{_{\M_F}}(\L)-\chi_{_{\M_F}}(1)}{\L-1}.\]
By writing the numerator as a Taylor polynomial in $(\L-1)$, we see that the Euler characteristic is $\chi_{_{\M_F}}(1) + \chi'_{_{\M_F}}(1).$
\end{proof}

\begin{example}
Let $X^\circ$ be a curve such that $\Trop(X^\circ)$ is a matroidal tropical variety.  Let $\Sigma$ be a graph structure on $\Trop(X^\circ)$.  Then $X$ corresponds to a family of curves degenerating to a union of rational curves.  Each vertex $v$ of $\Sigma$ contributes $(\L-E(v)+1)$ to the tropical motivic nearby fiber where $E(v)$ is the number of edges containing $v$.  Each bounded edge contributes $-\L+1$ while each unbounded edge contributes $1$.  Let $V,B,U$ denotes the number of vertices, bounded edges, and unbounded edges, respectively.  Then,
\[\psi_{\Trr}=\left(\sum_v (\L-E(v)+1)\right)+B(-\L+1)+U=\L(v-B)+(v-B)=(\L+1)\chi(\Sigma).\]
Specializing to $\L=1$, we get $e(X_{\gen})=2\chi(\Sigma)$, and hence $X_{\gen}$ is a smooth curve of genus $h^1(\Sigma)$. 
\end{example}

\section{Open problems}\label{s:open}

We briefly mention some open problems and directions for further research.

\begin{enumerate}

\item Can one give a geometric interpretation of the tropical motivic nearby fiber $\psi_{(X^\circ, \tri)}$ when $(X^\circ, \P)$ is a tropical pair, but $X^\circ$ is not necessarily sch\"on? What about its evaluation $e(\psi_{(X^\circ, \tri)})$ under the specialization $K_0(\Var_\C) \rightarrow \Z$ which takes the class of a variety $V$ to its Euler characteristic $e(V)$?

\item To what extent can one relax the condition of `smoothness' and replace it with `orbifold singularities'? 
More specifically, can one associate a limit mixed Hodge structure to a family $f: X \rightarrow \D^*$ over the punctured disk in which the fibers have at worst orbifold singularities? Can one compute a `motivic nearby fiber' given a semi-stable degeneration in which the central fiber is reduced and has `orbifold normal crossings'?

\item When $\Trop(X^\circ)$ is matroidal, can one use Proposition~\ref{p:duck} to explicitly write down all terms and maps in the corresponding 
Steenbrink spectral sequence  \cite[Corollary~4.20]{SteLimits}, and hence deduce a reasonable combinatorial formula for  the corresponding limit mixed Hodge numbers? Can one give interesting examples of matroidal tropical varieties which are not curves or hypersurfaces?

\item
What is the combinatorial significance of  property \eqref{e:hard} of the logarithm $N$ of the monodromy map when $X^\circ$ is a sch\"on hypersurface? That is, what is the significance of the fact that the 
sequences $\{ h^{p + i,i}(H^{d - 1}(X_\infty))  \mid 0 \le i \le d - 1 - p \}$ are symmetric and unimodal
for $0 \le p \le d - 1$?

\end{enumerate}

\appendix

\section{The Hodge-Deligne polynomial of a non-degenerate hypersurface of a torus}\label{s:hdp}

The goal of this section is to state a formula for the Hodge-Deligne polynomial of a non-degenerate hypersurface $Z$ of a torus with Newton polytope $P$, as it appears in \cite{BMString}. Throughout, $P$
is a $(d + 1)$-dimensional lattice polytope in a lattice $M$.  

If $B$ is a finite poset, then
the M\"obius function $\mu_{B}: B \times B \rightarrow \mathbb{Z}$ is defined recursively as follows,
\[
\mu_{B}(x,y) =  \left\{ \begin{array}{ll}
1 & \textrm{  if  } x = y \\
0 & \textrm{  if  } x > y \\
- \sum_{x < z \leq y} \mu_{B}(z,y) = - \sum_{x \leq z < y} \mu_{B}(x,z) & \textrm{  if  } x < y 
\end{array} \right.
.
\]
For any pair $z \leq x$ in $B$, we can consider the interval $[z,x] = \{ y \in B \mid z \leq y \leq x \}$. 
Suppose that $B$ has a minimal element $0$ and a maximal element $1$, and 
that every maximal chain in $B$ has the same length. The \emph{rank} $\rho(x)$ of an element $x$ in $B$ is equal to the length of a maximal chain in $[0,x]$, and the rank of $B$ is $\rho(1)$. In this case, we say that $B$ is \emph{Eulerian} if $\mu_{B}(x,y) = (-1)^{\rho(x) - \rho(y)}$ for $x \leq y$. 

\begin{remark}\label{evenodd}
Alternatively, one verifies that $B$ is Eulerian if and only if every interval of non-zero length contains as many elements of even rank as odd rank. 
\end{remark}

\begin{example}
The poset of faces of a polytope $P$ (including the empty face) is an Eulerian poset under inclusion. 
If we consider the empty face to have dimension $-1$, then $\rho(Q) = \dim Q + 1$, for any face $Q$ of $P$. 
\end{example}

\begin{example}
The Boolean algebra on $r$ elements consists of all subsets of a set of cardinality $r$ and forms an Eulerian poset under inclusion. 
\end{example}

The $G$-polynomial of an Eulerian poset was introduced by Stanley. 

\begin{definition}\cite{StaGeneralized}
If $B$ is an Eulerian poset of rank $n$, then 
\begin{equation*}
G(B,t) =  \left\{ \begin{array}{ll}
1 & \textrm{  if  } n = 0 \\ -
\tau_{ < n/2 } (\sum_{0 < x \leq 1} (t - 1)^{\rho(x)} G([x,1],t))  & \textrm{  if  } n > 0
\end{array} \right.
,
\end{equation*} 
where $\tau_{ < n/2 } $ is the truncation map which takes a polynomial and associates all the terms of degree less than $n/2$. 
\end{definition}

\begin{example}
If $B$ is the Boolean algebra on $r$ elements, then one verifies that $G(B,t) = 1$. 
\end{example}

If $Q$ is a face of $P$, then 
recall that the \emph{Ehrhart polynomial} $f_Q(m)$ of $Q$  is defined by 
$f_Q(m) = \#(mQ \cap M)$,
for each positive integer $m$, and its generating series has the form
\[
\sum_{m \ge 0} f_Q(m)t^m = \frac{h_Q^*(t)}{(1 - t)^{\dim Q + 1}},
\]
where $h_Q^*(t)$ 
 is a polynomial of degree at most $\dim Q$ with non-negative integer coefficients (see, for example, \cite{BRComputing}). 
If $Q$ is the empty face of $P$, then we set $h_Q^*(t) = 1$. 


\begin{definition}\cite{BMString}
If $Q$ is a (possibly empty) face of $P$, then 
\begin{equation*}
\tilde{S}(Q, t) = \sum_{Q' \subseteq Q} (-1)^{\dim Q - \dim Q'}h^*_{Q'}(t) G([Q',Q],t). 
\end{equation*}
\end{definition}

We can now present Batyrev and Borisov's formula for $E(Z;u,v)$ as stated in \cite{BMString}. We remark that an alternative formula was earlier obtained by Khovanski{\u\i} in unpublished work. If $B$ is an Eulerian poset, then let $B^*$ denote the Eulerian poset with all orderings between elements reversed.

\begin{theorem}\label{TheoremA}\cite{BBMirror, BMString}
If $Z$ is a non-degenerate hypersurface with respect to a lattice polytope $P$ of dimension $d + 1$, 
then 
\begin{equation*}
E(Z;u,v) = (1/uv)[(uv -1)^{d + 1} + (-1)^{d} \sum_{Q \subseteq P} u^{\dim Q + 1}
\tilde{S}(Q, u^{-1}v)G([Q,P]^{*},uv)].
\end{equation*}
\end{theorem}

\bibliographystyle{amsplain}
\def\cprime{$'$}
\providecommand{\bysame}{\leavevmode\hbox to3em{\hrulefill}\thinspace}
\providecommand{\MR}{\relax\ifhmode\unskip\space\fi MR }
\providecommand{\MRhref}[2]{%
  \href{http://www.ams.org/mathscinet-getitem?mr=#1}{#2}
}
\providecommand{\href}[2]{#2}

\end{document}